\documentclass[12pt]{amsart}
\usepackage{amssymb,amsfonts,amsthm,amsmath}
\usepackage{bbm}        
\usepackage{colordvi}
\usepackage{xcolor}
\usepackage{graphicx}
\usepackage{mathrsfs}   
\usepackage{mathtools}
\usepackage{vmargin}
\setpapersize{USletter}
\setmargrb{2.5cm}{2.5cm}{2.5cm}{2.5cm} 

\numberwithin{equation}{section}
\newtheorem{theorem}{Theorem}[section]
\newtheorem{proposition}[theorem]{Proposition}
\newtheorem{lemma}[theorem]{Lemma}
\newtheorem{corollary}[theorem]{Corollary}
\newtheorem{conjecture}[theorem]{Conjecture}
\theoremstyle{remark}
\newtheorem{definition}[theorem]{Definition}
\newtheorem{example}[theorem]{Example}
\newtheorem{remark}[theorem]{Remark}
\newcounter{FNC}[page]
\renewcommand{\Magenta}[1]{{\color{magenta}{#1}}}
\def\fauxfootnote#1{{\addtocounter{FNC}{2}\Magenta{$^\fnsymbol{FNC}$}%
     \let\thefootnote\relax\footnotetext{\Magenta{$^\fnsymbol{FNC}$#1}}}}
\def\cprime{$'$}

\newcommand{\CC}{{\mathbb C}}
\newcommand{\NN}{{\mathbb N}}
\newcommand{\PP}{{\mathbb P}}
\newcommand{\QQ}{{\mathbb Q}}
\newcommand{\RR}{{\mathbb R}}
\newcommand{\TT}{{\mathbb T}}
\newcommand{\ZZ}{{\mathbb Z}}
\newcommand{\bOne}{\mathbbm{1}}

\newcommand{\pr}{{p}{r}}

\newcommand{\Berg}{{\mathcal Ber}}
\newcommand{\calF}{{\mathcal F}}
\newcommand{\calH}{{\mathcal H}}
\newcommand{\calL}{{\mathcal L}}
\newcommand{\calO}{{\mathcal O}}
\newcommand{\calX}{{\mathcal X}}
\newcommand{\calV}{{\mathcal V}}

\newcommand{\scrP}{\mathscr{P}}
\newcommand{\scrA}{\mathscr{A}}
\newcommand{\coscrA}{{\it c}\mathscr{A}}
\newcommand{\trop}{\mathscr{T}}
\newcommand{\pls}{\mathscr{P}^{\infty}}

\newcommand{\TOP}{{\rm top}}
\DeclareMathOperator{\ini}{{\rm in}}
\DeclareMathOperator{\Log}{{\rm Log}}
\DeclareMathOperator{\Arg}{{\rm Arg}}
\DeclareMathOperator{\Hom}{{\rm Hom}}

\newcommand{\defcolor}[1]{\Cyan{#1}}
\newcommand{\demph}[1]{\defcolor{{\sl #1}}}

\title{Phase limit sets of linear spaces and discriminants}
\author{Mounir Nisse}
\address{Mounir Nisse\\
   Department of Mathematics\\
   Xiamen University Malaysia\\
   Jalan Sunsuria, Bandar Sunsuria\\
   43900, Sepang, Selangor\\ Malaysia}
\email{mounir.nisse@gmail.com}
\urladdr{https://sites.google.com/site/mounirnisse1/}
\thanks{Research supported by Xiamen University Malaysia Research Fund (Grant no.\ XMUMRF/2020-C5/IMAT/0013)}
\author{Frank Sottile}
\address{Frank Sottile\\
         Department of Mathematics\\
         Texas A\&M University\\
         College Station\\
         Texas \ 77843\\
         USA}
 \email{sottile@math.tamu.edu}
\urladdr{www.math.tamu.edu/\~{}sottile}
\thanks{Research supported in part by NSF grants DMS-1501370, DMS-2201005, and Simons  Foundation Collaboration Grant for
  Mathematics 636314.} 
\begin{document}

\begin{abstract}
 We show that the closure of the coamoeba of a linear space/hyperplane complement is the union of products of coamoebas of hyperplane
 complements coming from flags of flats, and relate this to the Bergman fan. 
 Using the Horn-Kapranov parameterization of a reduced discriminant, this gives a partial description of the phase limit sets of
 discriminants and duals of toric varieties.
 When $d=3$, we show that each 3-dimensional component of the phase limit set of the discriminant is a 
 prism over a discriminant coamoeba in dimension  $2$, which has a polyhedral description by a result of Nilsson and Passare.
\end{abstract}

\maketitle
\section*{Introduction}

Gelfand, Kapranov, and Zelevinsky~\cite{GKZ} defined the amoeba of a subvariety $X$ of a complex torus
to be the set of lengths (coordinatewise logarithm of absolute value) of points in $X$. 
Previously, Bergman~\cite{Berg} and Bieri and Groves~\cite{BiGr} showed that the asymptotic directions of the
amoeba form a polyhedral complex in the sphere.
The cone over this {\sl logarithmic limit set} is now known as the tropical variety of $X$~\cite{Stu_CBMS}.

Passare (in a talk in 2004) defined the coamoeba of $X$ to be the set of angles (coordinatewise arguments) of points in $X$.
Its phase limit set~\cite{NS} captures the asymptotic behavior of the angles in $X$.
This has a combinatorial structure reflected in the tropical variety and the closure of a coamoeba is its union with its phase
limit set.
Coamoebas independently arose in the study of mirror symmetry and dimer models~\cite{FHKV}, have found further
applications~\cite{FU} in this area, and were used to study Euler-Mellin 
integrals~\cite{BFP}.
The phase limit set was used to prove a version of Chow's Lemma for subvarieties of the torus~\cite{MN}, and
to prove that complements of coamoebas are higher convex~\cite{NS15} in the sense of 
Gromov~\cite[Sect.~$\tfrac{1}{2}$]{Gromov}.
The key idea is that the phase limit set forms a combinatorial skeleton of the coamoeba.

The intersection of a linear space $V$ in $\CC^n$ with the torus $(\CC^\times)^n$ is the complement $\calH^c:= V \setminus \calH$ of a
central arrangement $\calH$ of $n$ hyperplanes in $V$.
We give (Theorem~\ref{Th:Linear}) a recursive description of the phase limit set of $\calH^c$ in terms of the matroid of $\calH$.
While such a description is expected, as the tropical variety of $\calH^c$ is described in terms of its matroid~\cite{AK,FS,Stu_CBMS},
our proof is direct and does not use a description of the tropical variety.

We use this description of the phase limit set of a hyperplane complement to study coamoebas of discriminants, extending
some results of~\cite{NP,PS} on discriminant coamoebas in dimension two to higher dimensions.
The Horn-Kapranov parameterization of the reduced discriminant~\cite{Kapranov} is the
composition of an inclusion of a $d$-dimensional linear space (actually the complement of a hyperplane arrangement) into
$(\CC^\times)^n$ followed by a homomorphism into $(\CC^\times)^d\subset\PP^d$.
The action of the torus of the toric variety enables us to describe the phase limit sets of both its dual variety and
of the reduced discriminant in terms of the phase limit set of the linear space.
We show that all $d$-dimensional components of the phase limit set admit a surjection to the closed coamoeba of a
discriminant of smaller dimension $d'<d$, and that for some, the fibres are the full compact torus $\TT^{d-d'}$.
Thus these components are prisms over $d'$-dimensional closed coamoeba.
When $d=3$, all top-dimensional components are prisms over 2-dimensional closed discriminant coamoebae
(Theorem~\ref{Th:conjecture}). 

A complete description of the coamoeba of a hyperplane complement remains open.
The dimension of the coamoeba (and also the amoeba) is determined by the matroid of the
arrangement $\calH$~\cite{DEPRY}.
The closure of the coamoeba is a semialgebraic set, as is its algebraic amoeba, and these have only been described for
hyperplanes and lines~\cite{NS19}.
Lastly, we expect (Conjecture~\ref{Conj:Main}) that the coamoeba of a discriminant in dimensions three and higher is a
subset of its phase limit set.
This is an analog of the solidity of discriminant amoebas~\cite{PST}.
Specifically, in dimension two, the discriminant coamoeba is the complement of a zonotope, and when $d=3$, this
conjecture and Theorem~\ref{Th:conjecture} imply that it is a finite union of prisms over such zonotope complements.

In Section~\ref{S:Beginnings}, we recall the definitions of amoebas, coamoebas, and tropical varieties, discuss the structure of the 
phase limit set of a variety, and establish how these objects transform under homomorphisms of tori.
In Section~\ref{S:LinearSpace}, we study the phase limit set of a linear space, proving Theorem~\ref{Th:Linear} which
describes the components of the phase limit set.
We also relate this to the  Bergman fan of the linear space.
In Section~\ref{S:discriminant} we recall the Horn-Kapranov parameterization of a reduced discriminant, which relates the
phase limit set of a linear space to that of a discriminant.
Corollary~\ref{Co:Coamoebae} describes the closed coameoba of the dual variety to a torus in terms of the coamoeba of a linear space.
Our structure theorems about phase limit sets of reduced discriminants are presented in Section~\ref{S:discriminantAmoeba}

\section{Coamoebas and phase limit sets}\label{S:Beginnings}

The compact torus $\defcolor{\TT}\subset\CC$ is the set of unit norm complex numbers.
The set \defcolor{$\CC^\times$} of nonzero complex numbers is isomorphic to the product $\RR\times\TT$ under
$z\mapsto(\log|z|,z/|z|)$.
Thus $(\CC^\times)^m\simeq \RR^m\times\TT^m$, and we let $\defcolor{\Log}\colon(\CC^\times)^m\to\RR^m$ and
$\defcolor{\Arg}\colon(\CC^\times)^m\to\TT^m$ be the two projection maps.
The \demph{amoeba} $\defcolor{\scrA(X)}\subset\RR^m$ and \demph{coamoeba} $\defcolor{\coscrA(X)}\subset\TT^m$
of a subvariety $X\subset(\CC^\times)^m$ are its images under $\Log$ and $\Arg$, respectively.
Given an affine variety $V\subset\CC^m$, we may also consider its amoeba and coamoeba, which are the amoeba and coamoeba
of its intersection $V\cap(\CC^\times)^m$ with the complement of the coordinate axes.

While these are images of algebraic varieties under real analytic maps, and thus are subanalytic sets, the amoeba is homeomorphic to a
semialgebraic set and the coamoeba is a semialgebraic set~\cite{BPR}---this is explained in~\cite{NS19}.

\begin{proposition}[{\cite{Joh17, Mik17}}] 
\label{P:dimension}
 The amoeba and coamoeba of a subvariety $X\subset(\CC^\times)^m$ have the same dimension as subanalytic subsets of\/ $\RR^m$ and $\TT^m$,
 respectively. 
\end{proposition}

\begin{proof}
 It suffices to show that the differentials of $\Log$ and $\Arg$ at $x\in X$ have the same rank.
 For $z\in(\CC^\times)^m$, the kernel of $d_z\Arg$ is the linear subspace $z\cdot\RR^m$ of
 $\CC^m=T_z(\CC^\times)^m$, the translation of $\RR^m$ by $z$.
 Similarly, the kernel of $d_z\Log$ is the linear space $z\cdot(\sqrt{-1}\,\RR^m)=\sqrt{-1}z\cdot\RR^m$.
 Thus $v\in T_xX$ lies in the kernel of $d_x\Log$ if and only if $\sqrt{-1}v$ lies in the kernel of $d_x\Arg$.
 The result follows as $T_xX$ is stable under multiplication by $\sqrt{-1}$.
\end{proof}

Work of Bergman~\cite{Berg} and of Bieri and Groves~\cite{BiGr} showed that the set of asymptotic directions of an
amoeba $\scrA(X)$ in $\RR^m$ forms a polyhedral complex in the sphere.
This \demph{logarithmic limit set} is the collection of all accumulation points of sequences 
$\{\Log(x_i)/\|\Log(x_i)\|\mid i\in\NN\}$ in the sphere where $\{x_i\mid i\in\NN\}\subset X$ is unbounded.
The cone over the logarithmic limit set is the \demph{tropical variety}
\defcolor{$\trop(X)$} of $X$~\cite{McS}.

The tropical variety has another concrete geometric interpretation.
The integer lattice $\ZZ^m$ in $\RR^m$ is identified with the lattice of cocharacters (one parameter subgroups) of
$(\CC^\times)^m$. 
Given a cocharacter $\alpha\colon \CC^\times\to(\CC^\times)^m$, we define a subvariety of $\CC^\times\times(\CC^\times)^m$, 
\[
   \defcolor{\calX(\alpha)}\ \vcentcolon=\ \{(t,x)\in \CC^\times\times(\CC^\times)^m \mid \alpha(t)\cdot x\in X\}\,.
\]
Projection to the first coordinate gives a family of varieties $\calX(\alpha)\to\CC^\times$ over $\CC^\times$.

A cocharacter $\alpha$ lies in the tropical variety $\trop(X)$ when $\calX(\alpha)$ is not closed in 
$\CC\times(\CC^\times)^m$, that is, when the scheme-theoretic limit
 \begin{equation}\label{Eq:initialScheme}
   \lim_{t\to 0} \calX(\alpha)_t\ \vcentcolon=\ \overline{\calX(\alpha)}_0
 \end{equation}
is nonempty, where \defcolor{$\calX(\alpha)_t$} is the fiber of $\calX(\alpha)$ over $t\in\CC^\times$
and $\overline{\calX(\alpha)}_0$ is the fiber over the point $0\in\CC$ of the closure of $\calX(\alpha)$ in $\CC\times(\CC^\times)^m$. 
Since the fiber $\calX(\alpha)_t$ equals $\alpha(t)^{-1}\cdot X$, our condition that $\calX(\alpha)$ is not closed
is equivalent to the limit of translates $\alpha(t)\cdot X$ as $t\to\infty$ being nonempty.
This limit~\eqref{Eq:initialScheme} is an \demph{initial scheme} of $X$ and is written
\defcolor{$\ini_\alpha X $}. 

Since we realize the initial scheme as the fiber of a variety that is the closure of a family of isomorphic varieties, we
may determine its dimension.

\begin{proposition}\label{P:initial_dimension}
  Suppose that $X$ is irreducible of dimension $d$ and that $\alpha\in\trop(X)$.
  Then every irreducible component of $\ini_\alpha X$ has dimension $d$.
\end{proposition}

With these definitions, the tropical variety may also be defined as the closure in $\RR^m$ of the rational
cone over those cocharacters $\alpha$ such that $\ini_\alpha X$ is nonempty,
\[
   \trop(X)\ =\ \overline{\bigcup\{ \QQ_{\geq0}\cdot\alpha\mid \ini_\alpha X\neq\emptyset\}}\,.
\]
The tropical variety admits the structure of a rational polyhedral fan~\cite{McS}, but this structure is not canonical.
Fix such a fan structure on $\trop(X)$.
Its cones $\sigma$ have the property that $\ini_\alpha X =\ini_\beta X$ when both  $\alpha$ and $\beta$ lie in the
relative interior of $\sigma$.
(The \demph{relative interior} of a cone $\sigma$ is the complement of all proper faces of $\sigma$.)
Write \defcolor{$\ini_\sigma X$} for this common initial scheme.
Moreover, the torus \defcolor{$\CC^\times_\sigma$} generated by the cocharacters in $\sigma$ acts freely on
$\ini_\sigma X$ and the orbit space is a subvariety of $(\CC^\times)^m/\CC^\times_\sigma$.
(This may be seen in~\eqref{Eq:initialScheme}.)
We also have that if $\tau,\sigma\in\trop(X)$ are cones with $\tau$ a face of $\sigma$, then
 \begin{equation}\label{Eq:facesOfFaces}
   \ini_\sigma(\ini_\tau X)\ =\ \ini_\sigma X\,.
 \end{equation}

While amoebas are always closed subsets of $\RR^m$ (the map $\Log$ is proper), coamoebas are typically not closed.
For an example, consider the coamoeba of the line $\ell$ in $(\CC^\times)^2$ defined by $x{+}y{+}1=0$.
In the fundamental domain $[-\pi,\pi]^2$ for $\TT^2$, $\coscrA(\ell)$ is the complement of the hexagon that is the
Minkowski sum of the segments between the origin and the points $(-\pi,0)$, $(0,-\pi)$, and $(\pi,\pi)$,
respectively, and the union of the vertices $(\pi,0)$, $(0,\pi)$, and $(\pi,\pi)$, which are the images of the points
of $\ell$ in the three quadrants of $\RR^2$ that meet $\ell$.
  \begin{equation}\label{Eq:two_triangles}
   \raisebox{-50pt}{%
    \begin{picture}(265,96)(-70,-12)
     \put(-75,38){$\coscrA(\ell):$}
     \put(-2,-2){\includegraphics{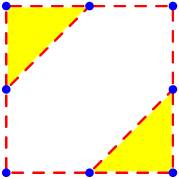}}
     \put(-25,-2){$-\pi$} \put(-15,38){$0$} \put(-15,77){$\pi$}
     \put(-10,-12){$-\pi$} \put(35,-12){$0$} \put(77,-12){$\pi$}
     \put(100,2){\vector(1,0){70}} \put(115,8){$\arg(x)$}
     \put(93,15){\vector(0,1){56}} \put(97,42){$\arg(y)$}
   \end{picture}}
  \end{equation}
The set-theoretic difference $\overline{\coscrA(\ell)}\smallsetminus\coscrA(\ell)$ is most of the three translated
subtori 
 \begin{equation}\label{Eq:subtori}
   (\pi,0)+(0,\pi)\cdot\TT\;,\quad
   (0,\pi)+(\pi,0)\cdot\TT\;,\quad \mbox{ and }\quad
   (\pi,0)+(-\pi,-\pi)\cdot\TT\;,
 \end{equation}
corresponding to the points where $\ell$ meets the coordinate lines in $\PP^2$.
These are the dotted lines in~\eqref{Eq:two_triangles}.

For a subvariety $X\subset(\CC^\times)^m$, the 
\demph{phase limit set $\scrP^\infty(X)$}
is the collection of all accumulation points of sequences of arguments $\{\Arg(x_i)\mid i\in\NN\}$ 
where $\{x_i\mid i\in\NN\}\subset X$ is unbounded.
By definition, the closure of the coamoeba is its union with its phase limit set,
\[
   \overline{\coscrA(X)}\ =\ \coscrA(X) \cup \scrP^\infty(X)\,.
\]
We will call $\overline{\coscrA(X)}$ the \demph{closed coamoeba} of $X$.
The phase limit set of the line $\ell$ with coamoeba~\eqref{Eq:two_triangles}  consists of the three
translated subtori~\eqref{Eq:subtori}.
Figure~\ref{F:LineCoamoeba} shows the coamoebas of two lines in $\CC^3$ from~\cite{NS}, in a fundamental domain
$[-\pi,\pi]^3$.
\begin{figure}[htb]
\[
  \includegraphics[height=110pt]{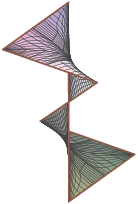}
   \qquad\quad
  \includegraphics[height=110pt]{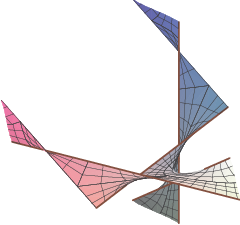}
\]
\caption{Coamoebas of lines in $\CC^3$.}
\label{F:LineCoamoeba}
\end{figure}
Their boundaries, which are their phase limit sets, consist of four translated subtori in the
directions $(\pi,0,0)$, $(0,\pi,0)$, $(0,0,\pi)$, and $(-\pi,-\pi,-\pi)$, which are lines in these
pictures\footnote{For motion pictures, see {\tt https://franksottile.github.io/research/stories/coAmoeba/index.html}.}.
These correspond to the four rays in the tropical variety of these lines.
These observations hold in general.

\begin{proposition}[{\cite[Theorem~1]{NS}}]\label{P:PLS}     
 The phase limit set of a subvariety $X$ of a torus is the union of the coamoebas of its initial schemes,
\[
   \scrP^\infty(X)\ =\  \bigcup\coscrA(\ini_\sigma X)\,.
\]
 This union is over all nonminimal cones $\sigma$ in the tropical variety $\trop(X)$ of $X$.
\end{proposition}

While this description refers to a particular fan structure on $\trop(X)$, the phase limit set does not depend upon the fan
structure.
We will fix a particular fan structure on $\trop(X)$.
To simplify discussion, we assume that the minimal linear space contained in $\trop(X)$, as a set, is its minimal cone, called
its \demph{lineality space}.
This is needed when discussing tropical compactifications.
Most other statements hold without this restriction.

The phase limit set of $X$ may be understood via the tropical compactification~\cite{Tevelev}
of $X$ given by the fan $\trop(X)$~\cite[Sect.~4.2]{NS}.
Let \defcolor{$Z$} be the toric variety associated to the fan $\trop(X)$~\cite{Fulton}.
For each cone $\sigma$ of $\Sigma$, there is an affine toric variety \defcolor{$V_\sigma$} (which is
$\mbox{spec}(\CC[\sigma^\vee\cap  M])$, where $M$ is the character lattice of $(\CC^\times)^m$ and $\sigma^\vee$ is the
cone dual to $\sigma$). 
The varieties $V_\sigma$ form a $(\CC^\times)^m$-equivariant open cover of $Z$.
Furthermore, $Z$ has a dense orbit and one orbit $\calO_\sigma\simeq (\CC^\times)^m/\CC^\times_\sigma$ for each cone
$\sigma$ of $\trop(X)$, where 
$\calO_\sigma$ is the minimal orbit of $V_\sigma$ and $\CC^\times_\sigma$ is the subgroup of $(\CC^\times)^m$ spanned by the
cocharacters in $\sigma$.
If $\tau\subset\sigma$ are cones of $\trop(X)$, then $\calO_\sigma$ lies in the closure of $\calO_\tau$.
The minimal cone \defcolor{$\mu$} of $\trop(X)$ is  its lineality space, and $X=\ini_\mu X$
carries an action of $\CC^\times_\mu$. 
The quotient $X/\CC^\times_\mu$ is naturally a subvariety of the dense orbit $\calO_\mu$ of $Z$, and its closure
$\overline{X}$ in $Z$ is a complete variety called the \demph{tropical compactification} of $X$.
(This is an actual compactification when $\mu=\{0\}$.)

For each cone $\sigma$ of $\trop(X)$, $\overline{X}\cap\calO_\sigma$ is identified with the quotient
$\ini_\sigma X/\CC^\times_\sigma$.
Moreover, if  $\tau\subset\sigma$ are cones of $\trop(X)$, then $\overline{X}\cap\calO_\sigma$ lies in the closure of
$\overline{X}\cap\calO_\tau$. 
Applying the map $\Arg$, we see that $\defcolor{\TT_\sigma}\vcentcolon=\Arg(\CC^\times_\sigma)$ acts on the coamoeba
$\coscrA(\ini_\sigma X)$   with quotient
\[
   \coscrA(\overline{X}\cap\calO_\sigma)\ \subset\ \coscrA(\calO_\sigma)\ =\ \TT^m/\TT_\sigma\,.
\]
Let $\defcolor{\pls_\sigma(X)}$ be the closure of $\coscrA(\ini_\sigma X)$.
This has an action of $\TT_\sigma$ with quotient the
closure of $\coscrA(\overline{X}\cap\calO_\sigma)$.
We call $\pls_\sigma(X)$  a \demph{prism} over $\coscrA(\overline{X}\cap\calO_\sigma)$ with fiber $\TT_\sigma$.
In~\eqref{Eq:two_triangles}, the dotted lines are prisms over points where the line meets the three
coordinate axes of $\PP^2$.

\subsection{Phase limit sets}

By Proposition~\ref{P:PLS}, the phase limit set $\pls(X)$ is the union of the $\pls_\sigma(X)$, for $\sigma$
not the lineality space $\mu$.
These sets $\pls_\sigma(X)$ are the \demph{strata} of the phase limit set $\pls(X)$.
If $\tau\subset \sigma$ are cones in $\trop(X)$,  then $\coscrA(\ini_\sigma X)$ is a subset of $\pls_\tau(X)$ as
$\ini_\sigma X$ is an initial scheme of $\ini_\tau X$~\eqref{Eq:facesOfFaces}.
A cone $\rho$ of $\trop(X)$ is a \demph{ray} if it is minimal among those properly contain the lineality space $\mu$.
We deduce a strengthening of Proposition~\ref{P:PLS}.

\begin{lemma}\label{L:rays}
  With these definitions, we have 
  \begin{equation}\label{Eq:components}
     \scrP^\infty(X)\ =\  \bigcup \pls_\rho(X)\,,
  \end{equation}
  the union over all rays $\rho$ of $\trop(X)$.
\end{lemma}

Not every ray $\rho$ of $\trop(X)$ is needed for the union~\eqref{Eq:components}.
For example, refining the fan of $\trop(X)$ may give new rays but does not change $\pls(X)$.

\begin{lemma}\label{L:refiningFan}
  Let $\Sigma$ be a fan refining a given tropical fan $\trop(X)$, and suppose that $\tau\in\Sigma$ is a cone that is not
  the lineality space. 
  Then there exists a cone $\sigma$ of $\trop(X)$ such that $\tau\subset\sigma$ and $\tau$ meets the relative interior of $\sigma$, 
  and $\pls_\tau(X)=\pls_\sigma(X)$.
\end{lemma}
\begin{proof}
  The existence (and uniqueness) of such a cone $\sigma$ is because $\Sigma$ refines $\trop(X)$.
  Any cocharacter $\alpha$ that lies in the relative interior of $\tau$ also lies in the relative interior of $\sigma$.
  Thus
  $\ini_\tau X = \ini_\alpha X  = \ini_\sigma X$, which implies the result.
\end{proof}

The tropical variety $\trop(\ell)$ of the line $\ell$ with coamoeba~\eqref{Eq:two_triangles} consists of
three rays, 
\[
   \RR_{\geq}(1,0)\,,\quad    \RR_{\geq}(0,1)\,,\quad \mbox{and}\quad   \RR_{\geq}(-1,-1)\,.
\]
The corresponding initial schemes are 
\[
   \ini_{(1,0)} \ell\ =\ \calV(y{+}1)\,,\quad
   \ini_{(0,1)} \ell\ =\ \calV(x{+}1)\,,\quad
   \ini_{(-1,-1)} \ell\ =\ \calV(x{+}y)\,.
\]
Their coamoebas are the three translated subtori of~\eqref{Eq:subtori}, which are prisms over points.

\begin{example}\label{Ex:plane}
 The closed coamoeba of the plane $\Pi$ defined by $x+y+z+1=0$ in $\CC^3$ is the complement in the fundamental
 domain $[-\pi,\pi]^3$ of the open zonotope that is the Minkowski sum of the four line segments from the origin to
 the points
\[
   (\pi,0,0)\,,\ (0,\pi,0)\,,\ (0,0,\pi)\,,\    (-\pi,-\pi,-\pi)\,.
\]
 These four segments are the primitive vectors of the rays in the tropical variety of $\Pi$, $\trop(\Pi)$.
 We show the closed coamoeba of $\Pi$, this zonotope, and $\trop(\Pi)$.
\[
  \includegraphics[height=1.5in]{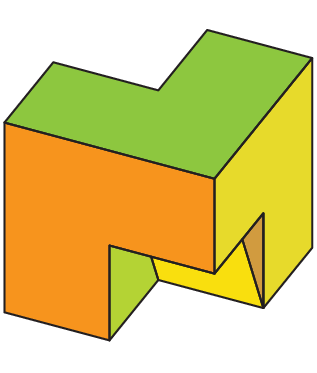}\quad\qquad
  \includegraphics[height=1.5in]{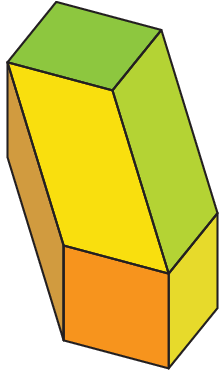}\quad\qquad
  \includegraphics[height=1.5in]{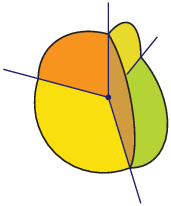}
\]
 For each ray $\rho$ in $\trop(\Pi)$, $\ini_\rho \Pi$ is irreducible and gives a stratum of the phase limit set $\pls(\Pi)$
 which is a prism over the coamoeba of a line in a plane.
 For example, $\ini_{(0,0,1)}\Pi=x{+}y{+}1$, so that the corresponding stratum $\pls_{(0,0,1)}(\Pi)$ (second from the right in
 Figure~\ref{F:cylinders}) is the product of the
 closure of the coamoeba~\eqref{Eq:two_triangles} and the subgroup $\TT_{(0,0,1)}$, which is a vertical segment in the fundamental
 domain, and so is three-dimensional.
 The other three strata are also prisms over coamoebas of lines.
 Figure~\ref{F:cylinders} shows these strata.
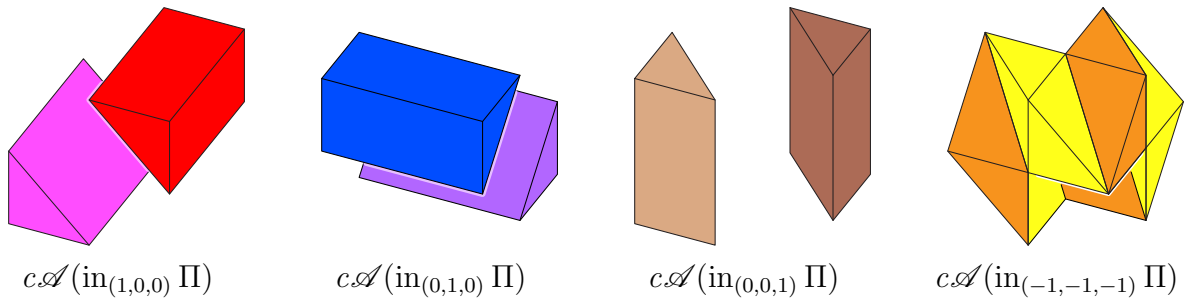
\begin{figure}[htb]
\centering
   \begin{picture}(91,106)(0,-15)
     \put(0,0){\includegraphics[height=1.27in]{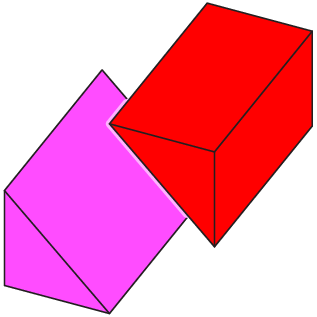}}
     \put(6,-15){$\coscrA(\ini_{(1,0,0)} \Pi)$}
    \end{picture}
     \qquad  
   \begin{picture}(91,106)(0,-15)
     \put(0,0){\includegraphics[height=1.27in]{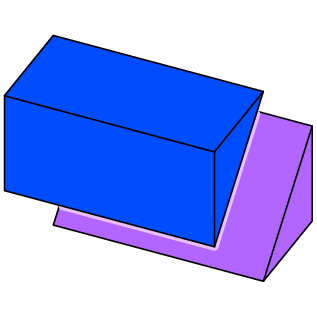}}
     \put(6,-15){$\coscrA(\ini_{(0,1,0)}\Pi)$}
    \end{picture}
     \qquad  
   \begin{picture}(91,106)(0,-15)
     \put(0,0){\includegraphics[height=1.27in]{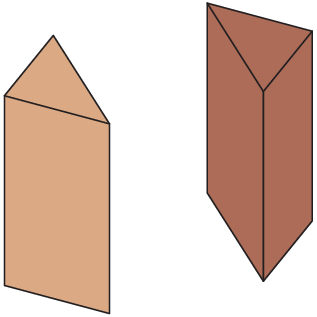}}
     \put(6,-15){$\coscrA(\ini_{(0,0,1)}\Pi)$}
    \end{picture}
     \qquad  
    \begin{picture}(91,105)(0,-15)
     \put(0,0){\includegraphics[height=1.27in]{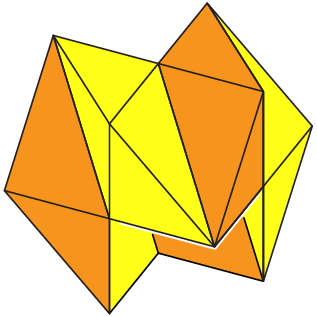}}
     \put(-4,-15){$\coscrA(\ini_{(-1,-1,-1)}\Pi)$}
    \end{picture}
 \caption{Strata of the phase limit set of a plane.}
 \label{F:cylinders}
 \end{figure}
 To understand the last, project the cube $[-\pi,\pi]^3$ along the vector $\RR(-1,-1,-1)$ to a hexagonal fundamental domain
 (in place of the square fundamental domain of~\eqref{Eq:two_triangles}).
 These strata cover $\coscrA(\Pi)$, with a general point of $\coscrA(\Pi)$ lying in two strata.
 If $\sigma$ is one of the six two-dimensional cones of $\trop(\Pi)$, then $\ini_\sigma\Pi$ is a translated orbit of the
 torus $\TT_\sigma$ so that $\pls_\sigma(\Pi)$ is two-dimensional.
 These form the boundaries of the strata in Figure~\ref{F:cylinders}. \hfill$\diamond$
\end{example}

Example~\ref{Ex:plane} illustrates a subtlety in the description of $\pls(X)$.
Both the coamoeba and phase limit set of $\Pi$ are covered by any three top-dimensional strata.

\begin{remark}\label{Rem:Polyhedra}
  This example also illustrates how the coamoeba of the plane $\Pi$ is a union of polyhedra, as its full-dimensional strata
  are each prisms over a coamoeba of a line.
 This plane $\Pi$ is also a (reduced) discriminant (see Example~\ref{Ex:Reduced_Discriminant}).
 Conjecture~\ref{Conj:Main} posits that any discriminant coamoeba is a subset of its phase limit set.
 By Theorem~\ref{Th:conjecture}, when $d=3$ each component of the phase limit set a union of
 polyhedra~\cite{NP}, so Conjecture~\ref{Conj:Main} implies that every discriminant coamoeba when $d=3$ is a
 finite union of  polyhedra. \hfill$\diamond$
\end{remark}

\subsection{Functoriality of tropical objects}

We explore how these tropical objects behave under maps of tori.
A homomorphism $\varphi\colon(\CC^\times)^m\to(\CC^\times)^n$ of tori 
induces a homomorphism $\varphi\colon \TT^m\to\TT^n$ and a linear map $\varphi\colon\RR^m\to\RR^n$ (we use the same notation for
all maps derived from $\varphi$).
These maps commute with $\Arg$ and $\Log$, respectively.
We collect some simple results.
The first three are well-known; they were used in~\cite{DFS,PS}.

\begin{proposition}\label{P:easyTropicalFunctoriality}
  Suppose that $X\subset(\CC^\times)^m$ is a subvariety and  $\varphi\colon(\CC^\times)^m\to(\CC^\times)^n$ is a
  group homomorphism.
  Let $Y$ be the closure of $\varphi(X)$, so that $\varphi\colon X\to Y$ is dominant.
 \begin{enumerate}
  \item The amoeba $\scrA(Y)\subset\RR^n$ is the closure of $\varphi(\scrA(X))$.
  \item The tropical variety $\trop(Y)$ equals  $\varphi(\trop(X))$. \label{easyTropicalFunctoriality.2}
  \item The image $\varphi(\coscrA(X))\subset\TT^n$ of the coamoeba of $X$ is dense in $\coscrA(Y)$ and
     we have 
    $\coscrA(Y) \subset \varphi\left(\overline{\coscrA(X)}\right)=  \overline{\varphi(\coscrA(X))}$. \label{easyTropicalFunctoriality.3}
  \item $\pls(Y)\subset\pls(X)$. \label{easyTropicalFunctoriality.4}
 \end{enumerate}
\end{proposition}
\begin{proof}
  Claim (1) follows from the continuity of $\Log$ and $\varphi$ as amoebas are closed.
  As a tropical variety is the cone over the logarithmic limit set, $(1)$ implies $(2)$.
  Claim (3) is a consequence of the continuity of $\Arg$ and $\varphi$, and that $\varphi\colon\TT^m\to\TT^n$ is proper.

  For (4), let $\theta\in\pls(Y)$.
  There exists an unbounded sequence $\{ y_i\mid i\in\NN\}\subset Y$ such that  $\lim_{i\to\infty} \Arg(y_i)=\theta$.
  Since $\overline{\varphi(X)}=Y$, for all $i\in\NN$ there exists a point $x_i\in X$ such that both distances
  between $\varphi(x_i)$ and $y_i$ and  between $\Arg(\varphi(x_i))$ and $\Arg(y_i)$ are at most $2^{-n}$.
  Then the sequence $\{x_i\mid i\in\NN\}\subset X$ is unbounded and $\lim_{i\to\infty} \Arg(\varphi(x_i))=\theta$.

  Let $\phi\in\TT^m$ be an accumulation point of $\{\Arg(x_i)\mid i\in\NN\}$ and replace $\{x_i\mid i\in\NN\}$ by a subsequence such that
  $\lim_{i\to\infty}\Arg(x_i)=\phi$.
  Then $\phi\in\pls(X)$ and $\varphi(\phi)=\theta$.  
\end{proof}

We prove a lemma about images of components of phase limit sets under $\varphi$.

\begin{lemma}\label{L:ConeLemma}
 Under the hypotheses of Proposition~\ref{P:easyTropicalFunctoriality}, suppose that we have fan structures on $\trop(X)$
 and $\trop(Y)$ such that $\varphi\colon \trop(X)\to\trop(Y)$ is a map of fans.
 For any cone $\tau$ of $\trop(Y)$, the stratum $\pls_\tau(Y)$ of $\pls(Y)$ is the union of images $\varphi(\pls_\sigma(X))$ of strata
 $\pls_\sigma(X)$ of  $\pls(X)$ such that $\varphi(\sigma)$ meets the relative interior of $\tau$.
\end{lemma}
\begin{proof}
  We use tropical compactification.
  Let \defcolor{$Z_X$} be the toric variety associated to the fan $\trop(X)$ and \defcolor{$Z_Y$} the toric variety
  associated to $\trop(Y)$.
  As $\varphi\colon \trop(X)\to\trop(Y)$ is  a map of fans, it induces a map $\varphi\colon Z_X\to Z_Y$ that is
  equivariant with respect to $\varphi\colon(\CC^\times)^m\to(\CC^\times)^n$.
  
  As $\varphi$ is a map of fans, whenever $\sigma$ is a cone of $\trop(X)$ and $\tau$ is a cone of $\trop(Y)$ such that
  $\varphi(\sigma)\subset\tau$, if $V_\sigma\subset Z_X$ and $V_\tau\subset Z_Y$ are the corresponding open sets of the tropical
  compactifications, then $\varphi(V_\sigma)\subset V_\tau$.
  When  $\varphi(\sigma)$ meets the relative interior of $\tau$ this is
  refined to $\varphi(\calO_\sigma)\subsetneq \calO_\tau$ and we have that 
  $\varphi(\CC^\times_\sigma)\subset\CC^\times_\tau$.

  In particular, if $\mu$ is the lineality space of $\trop(X)$ and $\nu$ the lineality space of $\trop(Y)$, then
  $\varphi(\mu)\subset\nu$ and $\varphi(\mu)$ meets the relative interior of $\nu$ (which is $\nu$).
  Thus the map $\varphi\colon(\CC^\times)^m\to(\CC^\times)^n$ of tori induces a map
  $\varphi\colon(\CC^\times)^m/\CC^\times_\mu\to(\CC^\times)^n/\CC^\times_\nu$, which is the map $\varphi\colon\calO_\mu\to\calO_\nu$ of
  dense orbits of $Z_X$ and $Z_Y$.
  Consequently, we have an induced map $\varphi\colon X/\CC^\times_\mu\to Y/\CC^\times_\nu$, which  is dominant.
  The tropical compactifications $\overline{X}$ and $\overline{Y}$ are the closures of these quotients in $Z_X$ and $Z_Y$, respectively.
  Thus $\varphi\colon\overline{X}\to\overline{Y}$ is a surjection.

  Let $\tau$ be a cone of $\trop(Y)$.
  By   Lemma~\ref{P:easyTropicalFunctoriality}\eqref{easyTropicalFunctoriality.2}, $\varphi\colon\trop(X)\to\trop(Y)$ is a
  surjection,
  so there is a cone $\sigma$ of $\trop(X)$ such that $\varphi(\sigma)\subset\tau$ and
  $\varphi(\sigma)$ meets the relative interior of $\tau$.
  Thus $\varphi(\calO_\sigma)\subset\calO_\tau$.
  Consequently, $\varphi(\overline{X}\cap\calO_\sigma)\subset \overline{Y}\cap\calO_\tau$.
  We have that  $\ini_\sigma X/\CC^\times_\sigma=\overline{X}\cap\calO_\sigma$ and 
  $\ini_\tau Y/\CC^\times_\tau=\overline{Y}\cap\calO_\tau$, so that
  $\varphi(\ini_\sigma X /\CC^\times_\sigma)\subset \ini_\tau Y/\CC^\times_\tau$.
  Since $\varphi(\CC^\times_\sigma)\subset\CC^\times_\tau$, we conclude that $\varphi(\ini_\sigma X)\subset\ini_\tau Y$.
  Since $\varphi\colon\overline{X}\to\overline{Y}$ is surjective, $\overline{Y}\cap\calO_\tau$ is the union of the images
  $\varphi(\overline{X}\cap\calO_\sigma)$ for such cones $\sigma$ of $\trop(X)$, and therefore
  $\ini_\tau Y$ is the union of the images $\varphi(\ini_\sigma X)$ for such cones $\sigma$.
  Applying $\Arg$ and taking closure completes the proof.
  %
  %
\end{proof}

We finish with a technical lemma that is needed for our description of the phase limit set of a discriminant.

\begin{lemma}\label{L:RayLemma}
 Under the hypotheses of Lemma~\ref{P:easyTropicalFunctoriality}, suppose that $\rho$ is a ray of $\trop(Y)$ that is
 not the image of any ray of $\trop(X)$.
 Let $\varphi^*(\rho)$ be the set of cones $\sigma$ of $\trop(X)$ such that $\rho$ meets the relative interior of $\sigma$.
 Then
 \begin{equation}\label{Eq:rayUnion}
   \pls_\rho(Y)\ \subset\ \bigcup_{\sigma\in\varphi^*(\rho)} \varphi( \pls_\sigma(X))\,.
 \end{equation}
\end{lemma}
\begin{proof}
  Let $\Sigma_X$ and $\Sigma_Y$ be the given fan structures on $\trop(X)$ and $\trop(Y)$.
  While $\varphi$ may not be a map of fans (the image of a cone in $\Sigma_X$ need not be contained in a single cone of $\trop(Y)$),
  there are refinements of the given fans so that $\varphi$ becomes a map of fans.
  Let us equip $\trop(X)$ and $\trop(Y)$  with such refinements.

  The ray $\rho$ remains a ray of $\trop(Y)$.
  If a cone $\tau$ of $\trop(X)$ has the property that $\varphi(\tau)$ meets the relative interior of $\rho$, then there is a
  cone $\sigma$ of $\varphi^*(\rho)$ such that $\tau$ meets   the relative interior of $\sigma$.
  Also, if $\sigma\in\varphi^*(\rho)$, then there is a cone $\tau$ of $\trop(X)$ that meets the relative interior of $\sigma$
  and has the property that $\varphi(\tau)$ meets the relative interior of $\rho$.
  Since, for any such pair of cones $\tau$ and $\sigma$,  $\pls_\tau(X)=\pls_\sigma(X)$, it is no loss to
  prove~\eqref{Eq:rayUnion} for the refinements $\trop(X)$ and $\trop(Y)$.
  But this follows by Lemma~\ref{L:ConeLemma}.  
\end{proof}

\section{The phase limit set of a linear space}\label{S:LinearSpace}

A \demph{multiset} is a set with some elements possibly repeated, so that $\{1,2,1\}=\{1,1,2\}$ and $\{1,2\}$ are distinct
multisubsets of $\ZZ$.
Let $V$ be a complex vector space with dual space $V^*$ and let $B\subset V^*$ be a finite
spanning multiset of nonzero vectors, which are linear forms on $V$.
We use $B$ and its subsets as indexing sets, so that $(\CC^\times)^B$ is the torus
$(\CC^\times)^{|B|}$ whose coordinates are indexed by the elements of $B$ and $\TT^B$ is its compact subtorus.
We extend this indexing to subsets of $B$.
Consider the injective map $\lambda_B\colon V\to\CC^B$ where
$\lambda_B(v)=( \langle b,v\rangle\mid b\in B)$, whose image is a linear subspace.
The tropicalization of its image (actually of $\lambda_B(V)\cap(\CC^\times)^B$) is the \demph{Bergman fan} of the
matroid associated to $B$~\cite{AK,FS}.
We first describe the phase limit set of this image and then discuss its relation to the Bergman fan.

\subsection{Coamoebas of hyperplane complements}\label{S:hyperplane}
Let \defcolor{$\calH_B$} or \defcolor{$\calH$} be the central arrangement of hyperplanes defined by the vectors in $B$.
We write \defcolor{$\calH^c_B$} or \defcolor{$\calH^c$} for the image of the hyperplane complement $V{\smallsetminus}\calH_B$ in 
$(\CC^\times)^B$  under the map $\lambda_B$. 

The set $\defcolor{F}\subset B$ of forms that vanish at a given point $x\in V$ is a \demph{flat} of $B$.
These are the coordinates of $\lambda_B(x)$ in $\CC^B$ which vanish.
Elements of the flat $F$  vanish along the intersection \defcolor{$L$} of the hyperplanes in
$\calH$ defined by forms in $F$.
We also call $L$ a \demph{flat} of $\calH$ (corresponding to $F$).
The forms in $F$ define a hyperplane arrangement \defcolor{$\calH/L$} in $V/L$, called the \demph{contraction} of $\calH$
by $L$. 
Restricting the forms $\defcolor{F^c}\vcentcolon=B\smallsetminus F$ not in $F$ (the hyperplanes not containing $L$)
to $L$ gives a hyperplane arrangement \defcolor{$\calH|_L$} in $L$, called the \demph{restriction} of $\calH$ to $L$.

For a flat $L$ of $\calH$ and the set $F\subset B$ of forms that vanish along $L$, we have injective maps
$\lambda_F\colon V/L\to\CC^F$ and $\lambda_{F^c}\colon L\to\CC^{F^c}$.
As with $\calH^c$, we consider the image \defcolor{$(\calH/L)^c$} of  $V/L\smallsetminus\calH/L$  in $(\CC^\times)^F$ under $\lambda_F$
and the image  \defcolor{$(\calH|_L)^c$} of $L\smallsetminus \calH|_L$ in $(\CC^\times)^{F^c}$ under $\lambda_{F^c}$.

\begin{example}\label{Ex:RunningI}
 Suppose that $V=\CC^3$ and let $B=\{b_1,\dotsc,b_6\}$ consist of the six row vectors of the matrix in Figure~\ref{F:RunningExampleI}.
 \begin{figure}[htb]
  \centering
${\displaystyle   \left( \begin{array}{ccc}
     1 & 0 & 0 \\
     0 & 1 & 0 \\
     0 & 0 & 1 \\
     1 & 2 & 0 \\
    -2 &-1 &-2 \\
     0 &-2 & 1 \end{array}\right)
}$  \qquad\qquad
   \raisebox{-46pt}{\begin{picture}(170,100)(-13,0)
     \put(0,0){\includegraphics{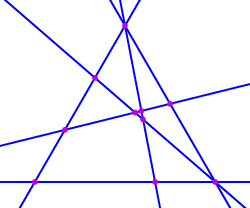}}
     \put(66,87){\small$p_{236}$}
     \put(105,17){\small$p_{124}$}
     \put(61,20){\small$\ell_6$}
     \put(3,37){\small$\ell_5$}
     \put(9,79){\small$\ell_4$}
     \put(42,76){\small$\ell_3$}
     \put(95,33){\small$\ell_2$}
     \put(36,2){\small$\ell_1$}
   \end{picture}}
 \caption{A vector configuration and corresponding line arrangement.}
 \label{F:RunningExampleI}
\end{figure}
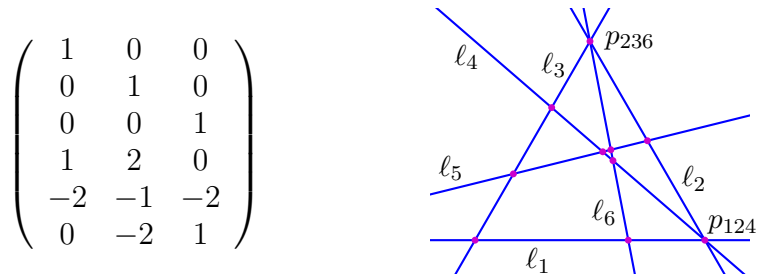
 These define a central hyperplane arrangement in $\CC^3$ whose corresponding line arrangement is shown in an affine patch
 of $\PP^2$ in Figure~\ref{F:RunningExampleI}.
 The line $\ell_i$ is defined by the vector $b_i$ of the $i$th row of the matrix.
 The remaining flats of $\calH$ are their intersections, which are points in this picture.
 Nine of the points are incident to two lines, while two ($p_{124},p_{236}$) are incident to three lines.
 The point $\ell_1\cap\ell_5$ is not shown.

 Consider the two flats, $\ell_1$ and $p_{124}$.
 Viewed projectively, both $\calH/\ell_1$ and $\calH|_{p_{124}}$ are points but
 $\calH|_{\ell_1}$ and $\calH/p_{124}$ are hyperplane (point) arrangements on a $\PP^1$:
 \[
  \calH|_{\ell_1}\ =\
   \begin{picture}(121,17)
   \put(0,0){\includegraphics{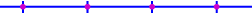}}
   \put(8,9){\small$\ell_5$}   \put(39,9){\small$\ell_3$}   \put(70,9){\small$\ell_6$}   \put(91,9){\small$\ell_2{=}\ell_4$}
 \end{picture}
 \qquad\qquad
 \calH/p_{124}\ =\ 
 \begin{picture}(90,17)
   \put(0,0){\includegraphics{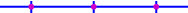}}
   \put(12,9){\small$\ell_1$}   \put(42,9){\small$\ell_4$}   \put(72,9){\small$\ell_2$}  
 \end{picture}
 \eqno{\diamond}
 \]
\end{example}\vspace{5pt}

Let us now study the phase limit set $\pls(\calH^c)$  of the hyperplane complement $\calH^c$.
Let $\theta\in\pls(\calH^c)$.
Then there is an unbounded sequence $\{z_i\mid i\in\NN\}\subset \calH^c$ such that 
$\lim_{i\to\infty} \Arg(z_i)=\theta$.
As $\lambda_B$ is an injection, there is a unique unbounded sequence 
$\{x_i\mid i\in\NN\}\subset V\smallsetminus\calH$ such that $\lambda_B(x_i)=z_i$ for every $i$.
Since $\lambda_B$ is linear, $\lambda_B(x_i/\|x_i\|)=z_i/\|x_i\|$ and $\Arg(z_i)=\Arg(z_i/\|x_i\|)$.
%
%
(For \defcolor{$\|\cdot\|$}, choose an inner product on $V$.)
Thus we may assume that the sequence $\{x_i\mid i\in \NN\}$ lies in the unit sphere of $V$.
As the sphere is compact, we may further assume that $\{x_i\mid i\in \NN\}$ converges to some 
point \defcolor{$x$} in the sphere, and thus in $V{\smallsetminus}\{0\}$.
Note that $x$ may lie in $\calH$.
We say that $\theta\in\pls(\calH^c)$ \demph{corresponds} to the point $x$.

We describe the set of points $\theta$ that correspond to $x\in V{\smallsetminus}\{0\}$.
To that end, let  $\{x_i\mid i\in\NN\}$ be a sequence in $V{\smallsetminus}\calH$ that converges to $x$ such that 
$\{\Arg(\lambda_B(x_i))\mid i\in\NN\}\subset\coscrA(\calH^c)$ converges to a point $\theta\in\TT^B$.
For each $i$, set $\defcolor{y_i}\vcentcolon=x_i-x$ and replace each $x_i$ by $x+y_i$.
Then $y_i\to 0$ and $x+y_i\in V{\smallsetminus}\calH$ for every $i$.
Let $\defcolor{F}\subset B$ be the set of linear forms in $B$ that vanish at $x$, which is a flat of $B$, and let
$F^c\vcentcolon=B\smallsetminus F$ be those linear forms that do not vanish at $x$.
Writing $\CC^B=\CC^F\times\CC^{F^c}$, $\TT^B=\TT^F\times\TT^{F^c}$, and  
$\lambda_B=\lambda_F\times\lambda_{F^c}$, we have 
\[
  \Arg(\lambda_B(x+y_i))\ =\
  \Arg(\lambda_F(y_i))\ \times\ \Arg(\lambda_{F^c}(x+y_i))\ \in\ \TT^F\times\TT^{F^c}\,.
\]
(As $\TT^F\times\TT^{F^c}=\TT^B$, we interpret this and similar assertions to mean that each factor on the left of the `$\in$' lies in the
corresponding factor on the right.)
Since $y_i\to 0$, but $x\neq 0$ is fixed, in the limit the second factor becomes $\Arg(\lambda_{F^c}(x))$.
For the limit of the first factor, observe that multiplying $y_i$ by any positive scalar
does not change $\Arg(\lambda_F(y_i))$, and also that $\lambda_F(y_i)$ depends not on $y_i$, but on $y_i$ modulo the
linear subspace $L$ of $V$ where the forms in $F$ vanish.
This linear space is the minimal flat of $\calH$ containing $x$, so that $x\in L\smallsetminus \calH|_L$ and
$\Arg(\lambda_F(y_i))\in\Arg((\calH/L)^c)$. 

Taking limits and considering all sequences $\{x_i\mid i\in\NN\}\subset V{\smallsetminus}\calH$ that converge to $x$ gives the following.

\begin{lemma}\label{Th:PLS_Linear}
  The set of points of $\pls(\calH^c_B)$ that correspond to a point $x\in V$ is the product
\[
   \overline{\coscrA((\calH/L)^c)}\ \times\  \{\Arg(\lambda_{F^c}(x))\}\ \subset\ \TT^F\times\TT^{F^c}\,,
\]
  where $L$ is the minimal flat of $\calH$ containing $x$ and $F^c$ are the forms that do not vanish at $x$.   
  For a flat $L$ of $\calH$, the union of these sets for points $x\in L{\smallsetminus}(\calH|_L)^c$ is
 \begin{equation}\label{Eq:PLS_hyperplane}
   \overline{\coscrA((\calH/L)^c)}\ \times\   \coscrA((\calH|_L)^c)
    \ \subset\ \TT^F\times\TT^{F^c}\,.
 \end{equation}
\end{lemma}
 
We use Lemma~\ref{Th:PLS_Linear} to give a description of the phase limit set of a linear space $V\subset\CC^{n}$.

\begin{theorem}\label{Th:Linear}
  Let $V\subset\CC^n$ be a linear space that does not lie in any coordinate plane, $\calH\subset V$ the central arrangement of
  hyperplanes given by the  coordinates of $\CC^n$, and $\calH^c\vcentcolon=V{\smallsetminus}\calH=V\cap(\CC^\times)^n$ its complement.
  Then, the phase limit set of $\calH^c$ in $\TT^n$ is 
  \begin{equation}\label{Eq:linearCoamoeba}
    \pls(\calH^c)\ =\ 
    \bigcup_L  \Bigl( \overline{\coscrA((\calH/L)^c)} \times\coscrA( (\calH|_L)^c)\Bigr)\,,
  \end{equation}
  the union over all nonzero proper flats $L$ of $\calH$, where $\calH/L$ is the  contraction of $\calH$ by $L$ and 
  $\calH|_L$ is the restriction of $\calH$ to $L$.
\end{theorem}

\begin{proof}
 The discussion preceding the statement of Lemma~\ref{Th:PLS_Linear} shows that every  point $\theta$ in the closure of
 the coamoeba of the hyperplane complement  $\calH^c$ corresponds to a point $x\in V$.
 The set of points corresponding to $x$ depends upon the minimal flat $L$ of $\calH$ containing $x$,
 so that $x\in L{\smallsetminus}\calH|_L$, 
 and~\eqref{Eq:PLS_hyperplane} describes all points $\theta$ corresponding to points in a given set $L{\smallsetminus}\calH|_L$.
 Thus the closure of the coamoeba of the hyperplane complement is the union of the sets~\eqref{Eq:PLS_hyperplane} over all
 flats, but this is the statement of Theorem~\ref{Th:Linear}.
\end{proof}

We refine the statement of Theorem~\ref{Th:Linear}.
Consider the first factor in the term of the decomposition~\eqref{Eq:linearCoamoeba} corresponding to a flat $L$,
$\overline{\coscrA((\calH/L)^c)}$, which is the union of the coamoeba of $(\calH/L)^c$ with its phase limit set, by
Proposition~\ref{P:PLS}.  
By Theorem~\ref{Th:Linear} this is the union over all nonzero flats $\Lambda$ of $\calH/L$ of a similar product, and the resulting term
of the decomposition~\eqref{Eq:linearCoamoeba} of $\overline{\coscrA(\calH^c)}$ gives the product
\[
  \overline{\coscrA((\calH/\Lambda)^c)}\ \times\ \coscrA(((\calH/L)|_\Lambda)^c)\ \times\ \coscrA( (\calH|_L)^c)\,.
\]
(We have $(\calH/L)/\Lambda = \calH/\Lambda$.)
These terms are indexed by flags of flats $\{0\}\neq L\subsetneq \Lambda$ of $\calH$.
Notice that when $\Lambda$ is a hyperplane, $(\calH/\Lambda)^c=\CC^n/\Lambda\smallsetminus\Lambda/\Lambda\simeq\CC^\times$,
whose coamoeba is isomorphic to $\TT$ and is closed. 

Continuing, we express the closure of the coamoeba as a union of products of coamoebas of restrictions of contractions.
Let
 \begin{equation}\label{Eq:flagOfFlats_subspaces}
   \defcolor{\calL}\ =\ (\{0\}=L_0\subsetneq L_1\subsetneq\dotsb\subsetneq L_k\subsetneq L_{k+1}=V)
 \end{equation}
be a flag of flats of $\calH$.
If $\defcolor{F_i}\subset B$ is the set of linear forms in $B$ that vanish along $L_i$, then
 \begin{equation}\label{Eq:flagOfFlats_subsets}
   \calF(\calL)\ =\ ( B = F_0\supsetneq F_1\supsetneq\dotsb\supsetneq F_k\supsetneq F_{k+1}=\emptyset)\,,
 \end{equation}
is a flag of flats  of $B$.

Let  $\calL$ be a flag of flats of $\calH$ with associated flag of flats $\calF(\calL)$~\eqref{Eq:flagOfFlats_subsets} of $B$.
Then $(\calH/L_{i-1})|_{L_i}$ is the hyperplane arrangement in $L_i/L_{i-1}$ induced by $F_{i-1}{\smallsetminus}F_i$. 
Let \defcolor{$((\calH/L_{i-1})|_{L_i})^c$} be the image of its complement in $(\CC^\times)^{F_{i-1}\smallsetminus F_i}$ under the map
$\lambda_{F_{i-1}\smallsetminus F_i}$.
Define 
\begin{equation}
  \label{Eq:calH(calL)}
   \defcolor{\calH(\calL)^c}\ \vcentcolon=\  \prod_{i=1}^{k+1} \bigl((\calH/L_{i-1})|_{L_i}\bigr)^c\,.
\end{equation}
Then its coamoeba in $\TT^B$ is
  \[
  \coscrA(\calH(\calL)^c)\ =\    \prod_{i=1}^{k+1} \coscrA\bigl(((\calH/L_{i-1})|_{L_i})^c\bigr)\
           \subset\ \prod_{i=1}^{k+1} \TT^{F_{i-1}\smallsetminus F_i}\ \subseteq\ \TT^B\,.
  \]
With these definitions, we give a refined version of the decomposition in Theorem~\ref{Th:Linear}.

\begin{corollary}\label{C:flagedDecomposition}
  As $\overline{\coscrA(\calH^c)}=\coscrA(\calH^c)\cup \pls(\calH^c)$, we have 
  \[
    \overline{\coscrA(\calH^c)}\ =\ \bigcup_{\calL} \coscrA(\calH(\calL)^c)\,,
  \]
  the union over all flags $\calL$ of flats in $\calH$.  
\end{corollary}

The closure of $\coscrA(\calH(\calL)^c)$ is the \demph{stratum} of $\overline{\coscrA(\calH^c)}$
corresponding to the flag of flats $\calL$.
We will show that this terminology is consistent with that in Section~\ref{S:Beginnings}.

\subsection{Relation to the tropical variety}

The tropicalization $\trop(\calH^c_B)$ of $\calH^c_B=\calH^c$ depends only on the  matroid of
$B$~\cite{Stu_CBMS}. 
We relate Corollary~\ref{C:flagedDecomposition} to the structure of $\trop(\calH^c_B)$.

A \demph{basis} of $B$ is a subset $E$ of $B$ that forms a basis for $V^*$.
Given a function $\omega\colon B\to\RR$, written $\omega\in\defcolor{\RR^B}$, a subset $S\subset B$ has a \demph{weight},
which is the sum of the values of $\omega$ at the elements of $S$.
Let \defcolor{$B_\omega$} be the set of bases of $B$ of maximal weight\footnote{This convention, while opposite
to~\cite{AK}, is necessary for compatibility with Section~\ref{S:Beginnings}.}. 
This forms a matroid on the ground set $B$.
A vector in $B$ is a \demph{loop} in $B_\omega$ if it does not lie in any basis in $B_\omega$.
When $B_\omega$ has no loops, $\omega$ induces a flag of flats \defcolor{$\calL(\omega)$}~\eqref{Eq:flagOfFlats_subspaces}
of $\calH$ and a corresponding flag of flats $\calF(\omega)$~\eqref{Eq:flagOfFlats_subsets} of $B$ with the following
property:
 \begin{equation}\label{Eq:omega}
   \mbox{$\omega$ is constant on each $F_{i-1}{\smallsetminus}F_i$ with
     $\omega|_{F_{i}\smallsetminus F_{i+1}} > \omega|_{F_{i-1}\smallsetminus F_i}$.}
  %
  %
 \end{equation}
Conversely, given any flag of flats $\calL$~\eqref{Eq:flagOfFlats_subspaces} of $\calH$ with corresponding flag
$\calF(\calL)$~\eqref{Eq:flagOfFlats_subsets} of flats of $B$, for any weight $\omega\in\RR^B$ satisfying~\eqref{Eq:omega},
we have $\calL=\calL(\omega)$ and $B_\omega$ has no loops. 
Furthermore, these $\omega$ with $\calL(\omega)=\calL$ form the relative interior of a nonempty cone
$\defcolor{\sigma_\calL}\subset\RR^b$.

Let \defcolor{$\trop(B)$} be the subset of $\RR^B$ consisting of those $\omega$ such that
$B_\omega$ contains no loops.
This  has a lineality space containing the vector $\defcolor{\bOne_B}\vcentcolon=(1,\dotsc,1)$, as
$B_\omega=B_{\omega+\bOne_B}$ for any $\omega\in\RR^B$.
This set $\trop(B)$ coincides with the tropical variety of the linear space $\lambda_B(V)\cap(\CC^\times)^B=\calH^c_B$, and
it admits several fan structures, which are studied in~\cite{AK,FS}.
In the \demph{fine subdivision} of $\trop(B)$ from~\cite{AK} each cone has the form $\sigma_\calL$, for some flag of flats $\calL$ of
$\calH_B$. 

We relate this to the varieties $\calH(\calL)^c$~\eqref{Eq:calH(calL)}.

\begin{lemma}\label{L:calLinitial}
  Let $\omega\in\ZZ^B$ be a weight lying in the relative interior of the cone $\sigma_\calL$.
  Then
\[
   \ini_\omega \calH^c_B\ =\ \calH(\calL)^c\,.
\]
\end{lemma}

This lemma gives interpretations for every term in the decomposition of $\overline{\coscrA(\calH_B^c)}$ of
Corollary~\ref{C:flagedDecomposition} in terms of cones in the fine subdivision of $\trop(B)$, and thus the corresponding terms from
Proposition~\ref{P:PLS}  and the strata of $\pls(\calH_B^c)$.

\begin{corollary}\label{C:fine_subdivision}
  For a cone $\sigma_\calL$ in the fine subdivision of $\trop(B)$, we have
  $\ini_{\sigma_\calL}\calH^c_B=\calH(\calL)^c$.
  When $\calL$ is not the flag $\{0\}\subset V$, we have
  $\pls_{\sigma_\calL}(\calH^c_B)=\overline{\coscrA(\calH(\calL)^c)}$.
  The closure of each term of the decomposition of the closure of the coamoeba in
  Corollary~\ref{C:flagedDecomposition} is a stratum of $\pls(\calH^c_B)$ for a 
  cone in the fine subdivision of $\trop(B)$.
\end{corollary}
\begin{definition}\label{D:componentsOfComplements}
  We simplify our notation, writing \defcolor{$\ini_\calL H^c_B$} for $\ini_{\sigma_\calL}\calH^c_B=\calH(\calL)^c$, and when
  $\calL$ is $\{0\}\subsetneq L\subsetneq V$ for a flat $L$, we write \defcolor{$\ini_L\calH^c_B$} for $\ini_\calL H^c_B$.
  We define \defcolor{$\pls_\calL(\calH^c_B)$} to be the stratum $\pls_{\sigma_{\calL}}(\calH^c_B)$ and write
  \defcolor{$\pls_L(\calH^c_B)$} when $\calL$ is the flag $\{0\}\subsetneq L\subsetneq V$.
\end{definition}
\begin{proof}[Proof of Lemma~\ref{L:calLinitial}]
  Suppose first that $k=1$ in~\eqref{Eq:flagOfFlats_subspaces}, so that the flag of
  flats $\calL$ is $\{0\}\subsetneq L\subsetneq V$.
  Set $F\vcentcolon=\{b\in B\mid b|_L\equiv 0\}$ and $F^c\vcentcolon=B\smallsetminus F$.
  Then the weight $\omega$ is constant on each of $F$ and $F^c$ with $a\vcentcolon= w|_F > w|_{F^c}=\vcentcolon b$.

  For $x\in V\smallsetminus\calH_B$ and $t\in\CC^\times$, we have
  \[
  \omega(t)\cdot\lambda_B(x)\ =\
  (t^a \lambda_F(x)\,,\, t^b\lambda_{F^c}(x))\ \in\ (\CC^\times)^F\times(\CC^\times)^{F^c}\,.
  \]
  Fix a splitting of the map $V\twoheadrightarrow V/L$ and use it to write $V=V/L\oplus L$.
  We may write a point $x\in V$ as $x=(y,z)$ with $y\in V/L$ and $z\in L$.
  Then
  \[
  \omega(t)\cdot\lambda_B(x)\ =\
  (t^a \lambda_F(y)\,,\, t^b\lambda_{F^c}(y) + t^b\lambda_{F^c}(z))\,,
  \]
  as $\lambda_F(L)=\{0\}$ and $z\in L$.

  Define an action $\nu$ of $\CC^\times$ on $V$ by $\nu(t).x=\nu(t).(y,z)=(t^{-a}y, t^{-b} z)$ for $t\in\CC^\times$.
  Then
  \[
  \omega(t)\cdot\lambda_B(\nu(t).x)\ =\ (\lambda_F(y),t^{b-a}\lambda_{F^c}(y)+\lambda_{F^c}(z))\,.
  \]
  As $b-a<0$,
  \[
   \lim_{t\to\infty} \omega(t)\cdot\lambda_B(\nu(t).x)\ =\ (\lambda_F(y),\lambda_{F^c}(z))\,.
  \]
  Since $\ini_\omega\calH^c_B=\lim_{t\to\infty} \omega(t)\cdot\calH^c_B$ and $\nu$ is a linear action of $\CC^\times$ on $V$, we have
  \[
  \ini_\omega\calH^c_B\ =\  \lim_{t\to\infty} \omega(t)\cdot\lambda_B(V\smallsetminus\calH_B)
  \ =\ \lambda_F(V/L\smallsetminus \calH/L)\times \lambda_{F^c}(L\smallsetminus \calH|_L)\ =\ \calH(\calL)^c\,.
  \]
  The case of a general flag of flats follows by an induction on $k$.
\end{proof}

There is a coarser subdivision of $\trop(B)$ which gives the \demph{Bergman fan}, written  \defcolor{$\Berg(B)$}.
Two weights $\omega,\nu\in\RR^B$ are \demph{equivalent} if they induce the same matroid, $B_\omega=B_\nu$.
Equivalence classes of weights form the relative interiors of the cones of the Bergman fan of $B$.

\begin{example}\label{Ex:BergmanFan}
  Consider the Bergman fan  for the matroid of Example~\ref{Ex:RunningI}.
  Let $e_1,\dotsc,e_6$ be the coordinate basis of $\RR^B=\RR^6$ with $e_i$ corresponding to $b_i$, that is, $e_i (b_j)=\delta_{ij}$. 
  We describe the structure of $\Berg(B)$ modulo $\RR\bOne_B$.
  For $I\subset\{1,\dotsc,6\}$ let \defcolor{$e_I\in \RR^B$} be the sum of $e_i$ for $i\in I$ and set $\defcolor{B_I}\vcentcolon=B_{e_I}$.
  Observe that
 \[
    B_{12}\ =\ \left\{ \{b_1,b_2,b_3\}\,,\; \{b_1,b_2,b_5\}\,,\; \{b_1,b_2,b_6\}\right\}\,.
 \]
 Since $b_4$ lies in no basis, it is a  loop in $B_{12}$.
 Notice that $B_{124}$  and  $B_{236}$ are loopless; for example
 $\{b_2,b_3,b_4\}$ is a basis for  $B_{124}$, as are the three bases in $B_{12}$ (and five others).
 Further observe that for $i=1,\dotsc,6$, $B_{i}$ consists of every basis that contains $b_i$ and it also has no loops.
 In fact the eight vectors $\omega\in\{e_{124}, e_{236}\,,\,e_1,\dotsc,e_6\}$ each give a ray of
 $\Berg(B)/\RR\bOne_B$.

 This fan $\Berg(B)/\RR\bOne_B$ is pure two-dimensional, and it has fifteen  maximal cones that occur in two types.
 One type are spanned by $e_i,e_j$ for $i\neq j$ and not both in $\{1,2,5\}$ or in $\{2,3,6\}$ and these correspond to the
 nine flats $p_{ij}=\ell_i\cap\ell_j$ of the line arrangement.
 Write \defcolor{$\sigma_{ij}$} for these nine cones.
 The second type consists of the six cones \defcolor{$\tau_{i,ijk}$} spanned by $e_i$ and $e_{ijk}$, for $\{i,j,k\}$ either $\{1,2,4\}$ or
 $\{2,3,6\}$. 

 If we intersect  $\Berg(B)/\RR\bOne_B$ with the sphere in $\RR^6/\RR\bOne_B$, we obtain a graph whose vertices correspond to rays of
 $\Berg(B)$ and edges to maximal cones.
 We show this graph in Figure~\ref{F:edge_graph}.
 \begin{figure}[htb]
   \centering
   \begin{picture}(201,151)(-100.5,-10)
   \put(-100.5,-7.2){\includegraphics{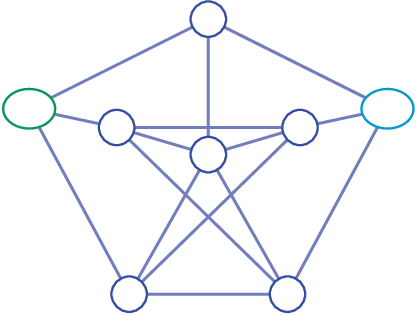}}

                                      \put(-5,133.5){\small $e_2$}
         \put(-63,119){\small$\tau_{2,124}$}                       \put(38,118){\small$\tau_{2,236}$}
         \put(-54,98){\small$\tau_{1,124}$}                       \put(31,98){\small$\tau_{3,236}$}
            \put(-53,94){\vector(-2,-1){10}}            \put(54,94){\vector(2,-1){10}}
         \put(3,107){\small$\sigma_{25}$}
                              \put(9,87){\small$\sigma_{13}$}
       \put(-95,90){\small $e_{124}$}                          \put(77,90){\small $e_{236}$}
   \put(-49,81){\small $e_1$}   \put(-4,67){\small $e_5$} \put(40.5,81){\small $e_3$}
       \put(-50,56){\small$\sigma_{15}$}  \put(-35,62){\vector( 2,1){20}}  
      \put( 36,56){\small$\sigma_{35}$}  \put( 36,62){\vector(-2,1){20}}
      \put(-86,38){\small$\tau_{4,124}$}                        \put(61.5,38){\small$\tau_{6,236}$}
           \put(-38,30){\small$\sigma_{45}$}                    \put(24,30){\small$\sigma_{56}$}
              \put(-20,14){\small$\sigma_{34}$}               \put(6,14){\small$\sigma_{16}$}
            \put(-42.5,0.5){\small $e_4$}  \put(33,0.5){\small $e_6$}
                         \put(-6,-5.5){\small$\sigma_{46}$}
  \end{picture}
   \caption{The Bergman fan as an graph.}\label{F:edge_graph}
 \end{figure}
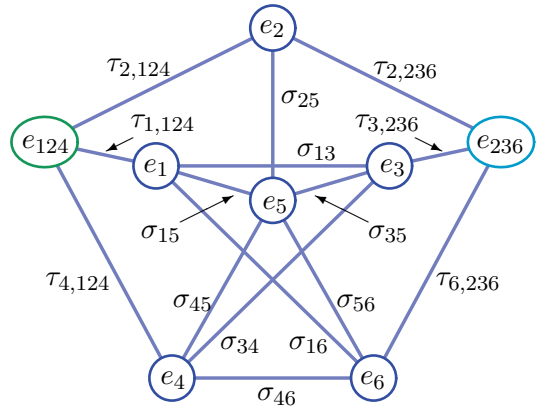
 In Example~\ref{Ex:Tdiscr} we will show a projection of this fan into $\RR^3$ where $e_i\mapsto b_i$.\hfill$\diamond$
\end{example}

By Lemma~\ref{L:rays}, the strata which cover $\pls(\calH^c_B)$ are among those indexed by the rays of the Bergman fan.
We will follow~\cite{FS} to describe the rays of the Bergman fan.
We begin by introducing some more structure of matroids.

\begin{definition}     \label{Def:Connectivity}
 A finite spanning set $B\subset V^*$ of nonzero vectors is \demph{disconnected} if there is a nontrivial direct sum
 decomposition $V^*=V^*_1\oplus V^*_2$ ($V^*_i\neq \{0\}$) with $B=(B\cap V^*_1)\cup (B\cap V^*_2)$.
 Otherwise $B$ is \demph{connected}.\hfill$\diamond$
\end{definition}

This terminology refers to the connectivity of the exchange graph of the matroid of $B$.
This is a graph with vertex set $B$ that has an edge between $b$ and $b'$ if there exists a basis $E$ of $B$ containing $b$ such
that $E{\smallsetminus}\{b\}\cup\{b'\}$ is also a basis.
The vertex set of a connected component of the exchange graph of $B$ forms a \demph{connected component} of $B$.

Given a finite spanning set $B\subset V^*$ of nonzero vectors, let $B_1,\dotsc,B_r$ be the connected components of $B$.
If for each $i$ we write $\defcolor{V^*_i}$ for the linear span of $B_i$, then we have the direct sum decomposition
$V^*=V^*_1\oplus \dotsb\oplus V^*_r$ and $B=B_1\sqcup\dotsb\sqcup B_r$.
Let $\defcolor{c(B)}\vcentcolon=r$ be the number of components of $B$.
If we set $\defcolor{V_i}\vcentcolon=\bigcap\{(V^*_j)^\perp\mid i\neq j\}$, then these subspaces of $V$ are in direct sum and
$V=V_1\oplus\dotsb\oplus V_r$. 
For each $i$, $B_i$ defines a hyperplane arrangement \defcolor{$\calH_{B_i}$} or \defcolor{$\calH_i$} in $V_i$ with complement
$\defcolor{\calH^c_i}\subset(\CC^\times)^{B_i}$.

We give some consequences of this definition.

\begin{lemma}\label{L:easyConsequence}
  Suppose that $B\subset V^*$ is a finite spanning set of nonzero vectors with connected components $B_1,\dotsc,B_r$ and corresponding
  decompositions $V^*=V^*_1\oplus \dotsb\oplus V^*_r$ and $V=V_1\oplus\dotsb\oplus V_r$ as above.
  \begin{enumerate}
   \item A set $E\subset B$ is a basis for $B$ if and only if for each $i$, $E_i\vcentcolon=E\cap B_i$ is a basis for $B_i$.
   \item The map $\lambda_B\colon V\to \CC^B$ is the product (direct sum) of maps
     $\lambda_{B_i}\colon V_i\to \CC^{B_i}$.
     We have the product decomposition of the hyperplane complement
   \[
        \calH^c_B\ =\ \calH^c_{B_1}\times \calH^c_{B_2}\times \dotsb \times  \calH^c_{B_r}\,,
   \]
    and a similar product decompositions for its coamoeba, amoeba, and Bergman fan.
  \end{enumerate}
\end{lemma}

The decomposition of $\calH^c_B$ from Lemma~\ref{L:easyConsequence}(2) is because  $\calH^c_B=\calH(\calL)^c$, where $\calL$ is any of the
$r!$ flags of flats in which each flat $L_j$ is a direct sum of some $j$ of the $V_i$.
In fact, any flag $\calL$ that coarsens such a flag is also trivial in that $\calH(\calL)^c=\calH^c_B$.
A consequence of Lemma~\ref{L:easyConsequence} is that to understand $\calH^c_B$ and its tropical objects it suffices to restrict ourselves
to the case when $B\subset V^*$ forms a connected matroid. 

Suppose that $B\subset V^*$ is a connected matroid.
A flat $L$ of $\calH$ is a \demph{flacet} if and only if the flat $\defcolor{F(L)}\vcentcolon=\{b\in B\mid b|_L\equiv 0\}$  and
its complement $\defcolor{F(L)^c}\vcentcolon=B{\smallsetminus}F(L)$ are connected matroids.
For the vector configuration/line arrangement of Example~\ref{Ex:RunningI}, the flacets are the lines (hyperplanes) and the two
points $p_{124}$ and $p_{236}$.
No other point $p_{ij}=\ell_i\cap\ell_j$ is a flacet, as $B_{p_{ij}}=\{b_i,b_j\}$ is a basis for $(V/p_{ij})^*$, showing
that it is disconnected. 
The lines $\ell$ are flacets because each line contains more than two points, for otherwise
$B\smallsetminus F(\ell)$ is a basis for $L^*$, where $L$ is the 2-plane corresponding to $\ell$.

Given a flacet $L$, let $\defcolor{\omega_L}\in\RR^B$ be the indicator function of $F(L)$.
This is the function $\omega_L\colon B\to\RR$ whose value at $b\in B$ is $1$ if $b|_L\equiv0$ and $0$ otherwise.
The purpose of this definition is the following.

\begin{proposition}[{\cite[Proposition~2.6]{FS}}]
  The rays of the Bergman fan $\Berg(B)$ of a connected matroid are generated by the indicator functions of its flacets.
\end{proposition}

In~\cite{FS}, the Bergman fan was studied as a subfan of the normal fan of the matroid polytope.
Its rays define facets, hence the term flacets.
By Lemma~\ref{L:rays}, we have the following.

\begin{corollary}
  When $B$ is connected, the phase limit set of $\calH^c_B$ is the union of the strata $\pls_\calL(\calH^c_B)$
  for flags  $\calL\colon\{0\}\subset L\subset V$, where $L$ is a flacet of $\calH_B$.
\end{corollary}

\section{The reduced discriminant}\label{S:discriminant}
We recall some fundamental properties of discriminants.
Much of this is found in~\cite{GKZ,Kapranov}.
Let $A\subset\ZZ^{m}$ be a multiset of $n$ vectors that span $\ZZ^{m}$ and lie in an affine hyperplane.
If $u$ is the primitive integer vector normal to that hyperplane, then $\langle u,a\rangle=1$ for all $a\in A$, as
$A$ spans $\ZZ^{m}$.
We adopt the notational convention that coordinates in $\ZZ^{n}$, $(\CC^\times)^{n}$, and $\CC^{n}$ are indexed by
elements of $A$, and we write $\ZZ^A$, $(\CC^\times)^A$, and $\CC^A$ for these spaces.

The surjective map of free abelian groups $\ZZ^A\to\ZZ^{m}$ defined by 
 \begin{equation}\label{Eq:lambdaA}
   \lambda_A\ \colon\ \ZZ^A\ni(c_a\mid a\in A)\ \longmapsto\ 
   \sum_{a\in A} c_a\cdot a\ \in\ \ZZ^{m}
 \end{equation}
has kernel a free abelian group of rank $\defcolor{d}\vcentcolon=n{-}m$.
Choose an isomorphism $\lambda_B$ between $\ZZ^d$ and this kernel.
This is given by $n$ vectors $\defcolor{B}=\{b_a\mid a\in A\}$ in the dual of $\ZZ^d$,
\[
    \lambda_B\ \colon\ \ZZ^d\ni y\ \longmapsto\ (\langle b_a,y\rangle\mid a\in A)\,.
\]
This collection $B$ is a \demph{Gale dual} to $A$.
Composing with $\lambda_A$, we obtain
\[
  \forall y\in\ZZ^d\qquad
  0\ =\  \sum_{a\in A} \langle b_a,y\rangle \cdot a\,.
\]
Pairing this with the primitive vector $u$ such that $\langle u,a\rangle=1$ for all $a\in A$ gives
\[
   0\ =\   \sum_{a\in A} \langle b_a,y\rangle \ =\ \Bigl\langle \sum_{a\in A}  b_a, y\Bigr\rangle\,,
\]
for all $y\in  \ZZ^d$, which implies that $0=\sum_{a\in A}b_a$.

\begin{example}\label{Ex:Reduced_Discriminant}
 ($m=1$)
 Let $A$ be the multisubset $\{1,\dotsc,1\}$ of $\ZZ$.
 The map $\lambda_A\colon\ZZ^{d+1}\to \ZZ$ is the sum of the coordinates,
 $\lambda_A(c_a\mid a\in A)=\sum_{a\in A} c_a$.
 A Gale dual $B$ to $A$ is given by the rows of the $(d{+}1)\times d$ matrix
 \begin{equation}\label{Eq:hyperplane}
  \left(\begin{array}{c} I_d\\-1\; \dotsb \; -1\end{array}\right)\ ,
 \end{equation}
 where $I_d$ is the $d\times d$ identity matrix.

 More interestingly, let $A$ be the columns of the matrix shown in Figure~\ref{F:running_Example_A}.
 \begin{figure}[htb]
   \centering
   $\left(\begin{array}{cccccc} 1&1&1&1&1&1\\0&0&1&4&2&3\\0&1&0&2&1&2\end{array}\right)$
   \qquad
   \raisebox{-29pt}{\begin{picture}(110,68)
     \put(0,0){\includegraphics{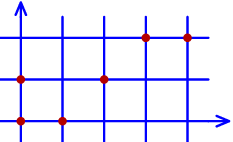}}
      \put(62,40){\small$6$}     \put(92,40){\small$4$}
      \put( 2,20){\small$2$}     \put( 52,20){\small$5$}
      \put( 2, 0){\small$1$}     \put( 32, 0){\small$3$}
   \end{picture}}
   \caption{Vectors in $A$ and their affine configuration.} \label{F:running_Example_A}
 \end{figure}
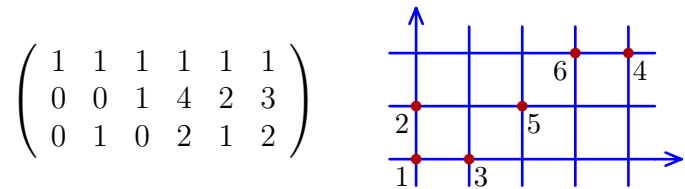
 This lies in the affine hyperplane where the first coordinate is 1.
 Figure~\ref{F:running_Example_A} shows the points of $A$ in that affine hyperplane.
 This is Gale dual to the configuration $B$ of Example~\ref{Ex:RunningI}.\hfill$\diamond$
\end{example}

An element $b\in\ZZ^d$ gives a group homomorphism $\CC^\times\to(\CC^\times)^d$ that sends $t\in\CC^\times$ to 
$\defcolor{t^b}\vcentcolon=(t^{b_1},\dotsc,t^{b_d})$.
Let $\varphi_B\colon(\CC^\times)^A\to(\CC^\times)^d$ be the group homomorphism defined by 
\[
   \defcolor{\varphi_B}(z_a\mid a\in A)\ \vcentcolon=\ \prod_{a\in A} z_a^{b_a}\ \in\ (\CC^\times)^d\,.
\]
An element $a\in\ZZ^{m}$ gives the monomial $\defcolor{x^a}\vcentcolon= x_1^{a_1}\dotsb x_m^{a_m}$, which is a character
of $(\CC^\times)^{m}$. 
The characters in $A$ give a group homomorphism $(\CC^\times)^{m}\to(\CC^\times)^A$ defined by 
 \begin{equation}\label{Eq:varphiA}
   \varphi_A\ \colon\ (\CC^\times)^{m}\ni x\ \longmapsto\ 
  (x^a\mid a\in A)\ \in\ (\CC^\times)^A\,.
 \end{equation}
These maps $\{1\}\to(\CC^\times)^{m}\xrightarrow{\varphi_A}(\CC^\times)^A\xrightarrow{\varphi_B}(\CC^\times)^d$ form an exact
sequence of groups which is obtained from the exact sequence 
$\ZZ^d\xrightarrow{\lambda_B}\ZZ^A\xrightarrow{\lambda_A}\ZZ^{m}\to\{0\}$ by applying $\Hom( \CC^\times,\bullet)$.

Let \defcolor{$Y_A$} be the image of $(\CC^\times)^{m}$ in $(\CC^\times)^A$ under $\varphi_A$.
The affine toric variety \defcolor{$X_A$} is its closure in $\CC^A$.
Writing \defcolor{$\check{\CC}^A$} for the dual space of hyperplanes in $\CC^A$, the \demph{dual variety}
$\defcolor{X^*_A}\subset\check{\CC}^A$ is the closure of the set of hyperplanes that contain a tangent space of $Y_A$. 
When $X^*_A$ is a hypersurface, its defining equation is the $A$-discriminant, and $A$ is \demph{nondefective}. 
There are several characterizations of nondefective configurations~\cite{CC07,Esterov,FI} whose equivalence is shown
explicitly in ~\cite{CD21}.
By~\cite[Thm.\ 18]{CC07}, the configuration $A$ is nondefective if and only if there is a complete flag of flats
$\emptyset\subsetneq F_1\subsetneq F_2\subsetneq\dotsb\subsetneq F_{d-1}\subsetneq B$ such that for all $j=1,\dotsc,d{-}1$
\begin{equation}\label{Eq:CC_char}
  \sum_{b\in F_j} b \quad\mbox{ does not lie in the span of $F_{j-1}$.}
\end{equation}
Call such a flag of flats in $B$ and its corresponding flag $\calH_B$ is a \demph{non-splitting flag}.
Observe that a non-splitting flag contains a flat of each possible rank $1,\dotsc,d{-}1$.
A flat $L$ of $\calH_B$ that lies in a non-splitting flag is a \demph{non-splitting flat}.
  
The group $(\CC^\times)^{m}$ acts on both $X_A$ and $X^*_A$ via the homomorphism $\varphi_A$.
While this has a single dense orbit in $X_A$, its orbits in $X^*_A$ are more interesting.
The closure in $\CC^d$ of the image \defcolor{$D_B$}  of $X^*_A\cap(\CC^\times)^A$ under the map $\varphi_B$ is a space of orbits of
$(\CC^\times)^{m}$ on $X^*_A$.
This is the \demph{reduced discriminant} of $A$.
When $A$ is nondefective, it is a hypersurface.

When $A$ is a pyramid, $X^*_A$ is contained in a coordinate hyperplane, $B$ contains a zero vector and $A$ is defective. 
We assume that $A$ is not a pyramid.
Let $\defcolor{Y^*_A} \subset X^*_A$ denote the dense open stratum consisting of all hyperplanes which do not contain
any coordinate axis, and which contain the tangent space of a point of $Y_A$.

The maps $\lambda_A\colon\ZZ^A\twoheadrightarrow\ZZ^{m}$ and $\lambda_B\colon\ZZ^d\hookrightarrow\ZZ^A$ of free abelian groups extend to
linear maps  $\lambda_A\colon\CC^A\twoheadrightarrow\CC^{m}$ and $\lambda_B\colon\CC^d\hookrightarrow\CC^A$, and this agrees with the
notation of Section~\ref{S:LinearSpace}.
Write $\defcolor{\calH_B}$ for the central hyperplane arrangement in $\CC^d$ obtained from pulling back the coordinate
hyperplanes of $\CC^A$ along $\lambda_B$.
Its hyperplanes are defined by the vectors in $B$.
Write \defcolor{$\calH_B^c$}  for the image $\lambda_B(\CC^d\smallsetminus\calH_B)\subset (\CC^\times)^A$ of its complement.

\begin{proposition}\label{P:RedDiscr}
  The dual variety $X^*_A$ is parameterized by $(\CC^\times)^{m}\times\calH^c_B$ via
 \begin{equation}\label{Eq:Param_Discr}
   (x,y)\ \longmapsto\ (x^a \cdot \langle b_a,y\rangle \mid a\in A)\,.
 \end{equation}
 The reduced discriminant $D_B$ is parameterized by the Horn--Kapranov uniformization $\psi$, which is the composition
 \begin{equation}\label{Eq:H-K}
   \psi\ \colon\ (\CC^d\setminus\calH_B)\ \xrightarrow{\ \lambda_B\ }\ (\CC^\times)^A\
     \xrightarrow{\ \varphi_B\ }\ (\CC^\times)^d\,.
 \end{equation}
 That is, the image of the Horn--Kapranov parameterization is a dense subset of $D_B$.
 It has dimension at most $d-1$.  
\end{proposition}
\begin{proof}
 We derive the parameterization~\eqref{Eq:Param_Discr} of $X^*_A$.
 The incidence variety
\[
   \defcolor{\Gamma}\ \vcentcolon=\ \{ (z,H)\in Y_A\times\check{\CC}^A\mid T_z Y_A \subset H\}\,,
\]
 has projections to $Y_A$ and $X^*_A$.
 The torus $(\CC^\times)^{m}$ acts on $\check{\CC}^A$, on $X^*_A$, and on $\Gamma$.

 Let $\defcolor{\bOne}\in(\CC^\times)^A$ be the identity of the torus.
 Thus $\Gamma$ is determined by its fiber \defcolor{$\Gamma_{\bOne}$} over ${\bOne}\in Y_A$, which is
 the set of hyperplanes containing the tangent space $T_{\bOne} Y_A$.
 Differentiating~\eqref{Eq:varphiA} shows that this tangent space is the image
 of $\CC^{m}$ under the  linear map
\[
   d_{\bOne}\varphi_A\ \colon\ 
    \CC^{m}\ni x\ \longmapsto\ (\langle x,a\rangle \mid a\in A)\,.
\]
 A hyperplane (given by the coefficients $(c_a\mid a\in A)$ of its defining linear form) lies in $\Gamma_{\bOne}$ if and
 only if its pullback to $\CC^{m}$ along this map $d_{\bOne}\varphi_A$ vanishes.
 This pullback is simply the map $\lambda_A$~\eqref{Eq:lambdaA}, and so the space of linear forms vanishing on 
 $T_{\bOne} X_A$ is $\lambda_B(\CC^d)$.
 Thus
\[
   \Gamma\ =\ \{ (\varphi_A(x), \varphi_A(x)\cdot\lambda_B(y))\mid
                x\in (\CC^\times)^{m}\ \mbox{ and }\ y\in\CC^d\}\,.
\]
 Projecting to the second coordinate parameterizes $X^*_A$, and gives~\eqref{Eq:Param_Discr}.

 To find the the image of $X^*_A$ under $\varphi_B$, we apply $\varphi_B$ to the parameterization~\eqref{Eq:Param_Discr}, obtaining
 \begin{equation}\label{Eq:imageRedDiscr}
  \prod_{a\in A} (x^a\cdot\langle b_a,y\rangle)^{b_a}\ =\ 
  \prod_{a\in A} (x^a)^{b_a}\cdot \prod_{a\in A} \langle b_a,y\rangle^{b_a}\ =\ 
  \prod_{a\in A} \langle b_a,y\rangle^{b_a}\,,
 \end{equation}
 as $\bOne=\prod_{a\in A} (x^a)^{b_a}$.
 This is because it is the image of $(\CC^\times)^{m}$ in $(\CC^\times)^d$ in the exact sequence
 $(\CC^\times)^{m}\xrightarrow{\,\varphi_A\,}(\CC^\times)^A\xrightarrow{\,\varphi_B\,}(\CC^\times)^d$.

 As $\sum_a b_a=0$, the map $\varphi_B$ is homogeneous of degree zero, so that for
 $x\in\lambda_B(\CC^d)\cap(\CC^\times)^A$ and 
 $t\in\CC^\times$, we have $t\cdot x\in\lambda_B(\CC^d)\cap(\CC^\times)^A$ and $\varphi_B(x)=\varphi_B(t\cdot x)$.
 Thus the image $\varphi_B(\lambda_B(\CC^d)\cap(\CC^\times)^A)$ has dimension at most $d{-}1$.
\end{proof}

\begin{remark}\label{R:translation}
 By~\eqref{Eq:Param_Discr}, $X^*_A$ is the closure of all translations of $\calH^c_B$ by the torus
 $(\CC^\times)^{m}$ acting on $(\CC^\times)^A$ through the map $\varphi_A$.
 By~\eqref{Eq:H-K}, the reduced discriminant $D_B$ is the image of the linear space
 $\lambda_B(V{\smallsetminus} \calH_B) = \calH^c_B$ under the map $\varphi_B$.
 \hfill$\diamond$
\end{remark}

\begin{example}\label{Ex:hypersurface}
 We consider the reduced discriminants from Example~\ref{Ex:Reduced_Discriminant}.
 When $A=\{1,\dotsc,1\}$ and $B$ consists of the rows of the matrix~\eqref{Eq:hyperplane}, the map  $\psi$ (where defined) is
 \[
  \psi\ \colon\ (y_1,\dotsc,y_d)\ \longmapsto\
  \left( -\frac{y_1}{y_1+\dotsb+y_d}\,,\, \dotsc\,,\,  -\frac{y_d}{y_1+\dotsb+y_d} \right)\ \in\ (\CC^\times)^d\, 
 \]
 The image is (a dense subset of) the hyperplane $x_1+\dotsb+x_d+1=0$.
 Thus the line with coamoeba~\eqref{Eq:two_triangles} and the plane of Example~\ref{Ex:plane} are reduced
 discriminants.

 For the configuration from Example~\ref{Ex:RunningI}, $\psi(x,y,z)$ is
 \[
 \left(\ \frac{x(x+2y)}{(-2x-y-2z)^2}\,,\,
        \frac{y(x+2y)^2}{(-2x-y-2z)(-2y+z)^2}\,,\,
        \frac{z(-2y+z)}{(-2x-y-2z)^2}\ \right)\,.
 \]
  This is the hypersurface in $\CC^3$ defined by
  \begin{equation}    \label{Eq:bigDiscrim}
  \begin{gathered}
 %
    {-}432p^6{-}1024p^5q^2{-}1152p^5q{+}864p^5r{+}216p^5{+}1280p^4q^2r{+}768p^4q^2{+}1584p^4qr\\
    {+}832p^4q{-}432p^4r^2{+}216p^4r{-}27p^4{-}40p^3q^2r^2{+}512p^3q^2r{-}192p^3q^2{+}1584p^3qr^2{+}532p^3qr\\
   {-}200p^3q{+}4000p^2q^3r^2{-}40p^2q^2r^3{+}5038p^2q^2r^2{-}208p^2q^2r{+}16p^2q^2{-}1152p^2qr^3{+}832p^2qr^2\\
   {-}200p^2qr{+}16p^2q{+}500pq^3r^3{-}200pq^3r^2{+}1280pq^2r^4{+}512pq^2r^3{-}208pq^2r^2\\
   {+}\Blue{3125q^4r^4{+}4000q^3r^4{-}200q^3r^3{+}16q^3r^2{-}1024q^2r^5{+}768q^2r^4{-}192q^2r^3
   {+}16q^2r^2}\makebox[2pt][l]{\mbox{\hspace{40pt}$\diamond$}} 
 \end{gathered}
  \end{equation}
\end{example}

 A hypersurface $Y=\calV(f)\subset(\CC^\times)^d$ has a map  \defcolor{$\gamma$} to $\PP^{d-1}$ called the 
 \demph{logarithmic Gauss map},
\[
    Y\ni y\ \longmapsto\  \gamma(y)\ \vcentcolon=\ 
    \Bigl[ y_1 \tfrac{\partial f}{\partial y_1}\,,\,
           y_2 \tfrac{\partial f}{\partial y_2}\,,\, \dotsc\,,\,
           y_d \tfrac{\partial f}{\partial y_d} \Bigr]\ 
     \in\ \PP^{d-1}\,.
\]
 Observe that as  $\sum_a b_a = 0$, the map $\varphi_B$ defines a rational map from the projective space
 $\defcolor{\PP^A}\vcentcolon=\PP(\CC^A)$ to $\CC^d$. 
 As $\lambda_B$ defines a linear injection $\PP^{d-1}\hookrightarrow\PP^A$, the Horn-Kapranov
 parameterization $\psi$ is also a rational map defined on $\PP^{d-1}$.

\begin{proposition}[Kapranov~\cite{Kapranov}]\label{P:Kapranov}
 When $A$ is nondefective, the map $\psi\colon\PP^{d-1}\,-\to\CC^d$ parameterizing the reduced discriminant $D_B$ is
 birational with inverse given by the logarithmic Gauss map.
 This property---that the logarithmic Gauss map is birational---characterizes reduced discriminants of nondefective
 configurations.
\end{proposition}

\begin{remark}\label{R:parallel}
  The \demph{parallel-equivalence class} of $b\in B$ is
  $\defcolor{[b]}\vcentcolon=\{ c\in B\mid \mbox{$c$ is parallel to $b$}\}$.
  Let \defcolor{$\eta\in\ZZ^d$} be a primitive vector parallel to $b$.
  For each $c\in [b]$, let \defcolor{$q_c$} be the nonzero integer such that $c=q_c\cdot \eta$.
  For $y\in (\CC^d\smallsetminus \calH_B)$, consider the contribution of vectors parallel to $b$ to the Horn-Kapranov
  parameterization~\eqref{Eq:H-K},
 \[
     \prod_{c\in [b]} \langle c,y\rangle^{c}\ =\ 
     \prod_{c\in [b]} (q_c\cdot \langle \eta,y\rangle)^{q_c\cdot \eta}\ =\ 
     \Bigl(\prod_{c\in [b]} q_c^{q_c}\Bigr)^\eta\;\cdot\;  \langle \eta,y\rangle^{\eta\,\cdot \sum q_c}\,.
 \]
  If $0=\sum q_c$, so that the sum of the vectors parallel to $b$ is $0$, then 
  this contribution is the element $(\prod_{c\in [b]} q_c^{q_c})^{\eta}\in (\ZZ^\times)^d$.
  Otherwise, it is a nonconstant function of $y$.\hfill$\diamond$
\end{remark}

Set $\defcolor{\TT_A}\vcentcolon=\varphi_A(\TT^{m})\subset\TT^A$.
We describe the closed coamoebas  of  $X^*_A$ and $D_B$.

\begin{corollary}\label{Co:Coamoebae}
 The closed coamoeba of the dual variety $X^*_A$ is  
\[
   \overline{\coscrA(X^*_A\cap(\CC^\times)^A)}\ =\ \TT_A\,\cdot\, \overline{\coscrA(\calH^c_B)}\,,
\]
and we have $ \overline{\coscrA(D_B)} = \varphi_B(\overline{\coscrA(\calH^c_B)}$ and
$\scrP^\infty(D_B) \subset \varphi_B(\scrP^\infty(\calH^c_B))$.
\end{corollary}

\begin{proof}
 By Remark~\ref{R:translation}, $X^*_A$ is the closure of $\varphi_A(\CC^\times)^{m}\cdot\calH_B^c\subset(\CC^\times)^A$.
 Thus the same holds for their coamoebas and closed coamoebas, which is the first statement.
 Proposition~\ref{P:easyTropicalFunctoriality}\eqref{easyTropicalFunctoriality.3} and~\eqref{easyTropicalFunctoriality.4}
 applied to the map $\varphi_B\colon X^*_A\to D_B$ implies the remaining statements.
\end{proof}

\begin{corollary}\label{Co:dim_Discr_Coamoeba}
  Suppose that $A$ is non-defective.
  When $d=1$ the reduced discriminant and its coamoeba are points.
  For $d>1$, the coamoeba discriminant has dimension $d$ in $\TT^d$.
\end{corollary}

\begin{proof}
  As $A$ is non-defective,  $D_B$ is an irreducible hypersurface in $(\CC^\times)^d$.
  Irreducible hypersurfaces in $\PP^1$ are points, which proves the first statement.
  Suppose that $d>1$.
  Either both the amoeba and coamoeba of $D_B$ are full-dimensional, or the defining equation of $D_B$ is an irreducible binomial,
  $y^a-cy^b=0$.
  (This is a consequence of the theorem on the dimension of an amoeba in~\cite{DRY}.)
  Suppose that $D_B$ is defined by $y^a-cy^b$.
  In $(\CC^\times)^d$ this becomes $y^\alpha=c$, where $\alpha=a-b$.
  Since $y_i\frac{\partial}{\partial y_i}y^\alpha=\alpha_i y^\alpha=c\alpha_i$, the logarithmic Gauss map on $D_B$
  is constant, and therefore $D_B$ is not a reduced discriminant when $d>1$, by Proposition~\ref{P:Kapranov}. 
\end{proof}
%

%
\section{Phase limit set of discriminants}\label{S:discriminantAmoeba}

Let $D_B$ be  a reduced discriminant arising from a nondefective configuration
$A\subset\ZZ^m$ with Gale dual $B\subset\ZZ^d$.
When $d\geq 2$ the discriminant coamoeba $\coscrA(D_B)$ has full dimension $d$, by Corollary~\ref{Co:dim_Discr_Coamoeba}.
When $d=2$, this coamoeba and its phase limit set were described and shown to be essentially
polyhedral objects~\cite{NP,PS}.
This is explained in Example~\ref{Ex:facial_discriminants}.

By Corollary~\ref{Co:Coamoebae}, we have $\pls(D_B) = \varphi_B(\pls(\calH^c_B))$.
Theorem~\ref{Th:Linear} and its Corollary~\ref{C:flagedDecomposition} describe the components of $\pls(\calH_B^c)$; we
study their images under $\varphi_B$ to understand $\pls(D_B)$.
Our main results are Theorems~\ref{Th:conjecture},~\ref{Th:main_components}, and~\ref{Thm:PLS_Discriminant}.

We assume that $d>2$.
By Theorem~\ref{Th:main_components}, the phase limit set $\pls(D_B)$ has $d$-dimensional strata that are prisms over
discriminant coamoebas corresponding to non-splitting flats of $\calH_B$.
In a different vein, Theorem~\ref{Thm:PLS_Discriminant}  identifies  the
$d$-dimensional components of the phase limit set $\pls(D_B)$ as the union of strata
$\pls_L(D_B)$ corresponding to certain essential flacets (see below) $L$ of $\calH_B$.
We do not know how to relate non-splitting flats and essential flacets, except when $d=3$, which gives Theorem~\ref{Th:conjecture}.
These results generalize what we observed for the plane in $\CC^3$ in Example~\ref{Ex:plane}.
That example and some computations suggest the following conjecture, which is an analog of the solidity of
discriminant amoebas~\cite{PST}.

\begin{conjecture}\label{Conj:Main}
  Let $A\subset\ZZ^m$ be nondefective with Gale dual $B\subset\ZZ^d$.
   When $d>2$, the closure of the discriminant coamoeba $\coscrA(D_B)$ equals the phase limit set of $D_B$. 
\end{conjecture}

In particular, the discriminant coamoeba is a subset of its phase limit set.
Equivalently, any fiber above  regular value of the argument map on $D_B$ is unbounded.
When $d=3$, we have an appealing structure theorem, which we prove in Section~\ref{SS:tropDiscr}.

\begin{theorem}\label{Th:conjecture}
  When $d=3$ the phase limit set of the discriminant $D_B$ is the union of prisms over closed coamoebas of 
  discriminants $D_{B|_H}$, for $H$ a hyperplane of $\calH_B$.
  If in addition Conjecture~$\ref{Conj:Main}$ holds, the closed discriminant coamoeba is this union of prisms.
\end{theorem}

\begin{remark}\label{R:prisms}
  The prisms in Theorem~\ref{Th:conjecture} all have a precise description using the constructions in  Section~\ref{SS:tropDiscr}
  and the results of~\cite{NP,PS}.
  Let $H\in\calH_B$ be a non-splitting hyperplane with $b\in B$ an element vanishing on $H$.
  Let $B|_H$ be the multiset of non-zero elements in the image of $B$ in $\ZZ^3/\ZZ b\simeq\ZZ^2$.
  By Lemma~\ref{L:factors}, $D_{B|_H}$ is a reduced discriminant.
  We have the map $\pr\colon\TT^3\to\TT^3/\TT_b=:\TT_H\simeq\TT^2$ and the closed coamoeba of $D_{B|_H}$ is a subset of $\TT_H$.
  Then the corresponding component of $\pls(D_B)$ is $\pr_H^{-1}(\overline{\coscrA(D_{B|_H})})$.

  Each discriminant coamoeba $\overline{\coscrA(D_{B|_H})}$ has a description as a polyhedral object~\cite{NP,PS}.
  We illustrate this in Example~\ref{Ex:facial_discriminants}.\hfill$\diamond$
\end{remark}

Our results use an understanding of strata of the phase limit set $\pls(D_B)$ in terms of the matroid of $B$.
This comes from two sources.
First, our understanding of  $\pls(\calH^c_B)$ from Section~\ref{S:LinearSpace} and the dominant map
$\varphi_B\colon\calH^c_B\to D_B$ leads to 
Theorem~\ref{Th:main_components} about the structure of strata of $\pls(D_B)$ corresponding to non-splitting flats of $\calH_B$.
Second, by Lemma~\ref{L:rays}, $\pls(D_B)$  is the union of strata corresponding to rays of the tropical variety \defcolor{$\trop(D_B)$}.
We observe that not all rays of $\trop(D_B)$ are needed and note that we only have a partial understanding of $\trop(D_B)$:
As it is the image of the Bergman fan under the map $\varphi_B\colon\RR^B\to\RR^d$, this map may be neither injective nor
surjective on rays. 
This leads to Theorem~\ref{Thm:PLS_Discriminant}, about rays that contribute $d$-dimensional strata to  $\pls(D_B)$.

Recall from Definition~\ref{D:componentsOfComplements} that for a flag of flats $\calL$ there is a cone
\defcolor{$\sigma_\calL$} in the fine subdivision of $\trop(B)$~\cite{AK} whose corresponding initial scheme
$\ini_\calL\calH^c_B$ is the set $\calH(\calL)^c$~\eqref{Eq:calH(calL)}.
Its closed coamoeba is the component $\pls_\calL(\calH^c_B)$ of the phase limit set of $\calH^c_B$.
We consider the image in $\TT^d$ under $\varphi_B$ of the coamoeba of such a component.
For a flat $L$, let \defcolor{$\pls_L(D_B)$} be the closed coamoeba of $\varphi_B(\ini_L\calH^c_B)$,
and for  a flag of flats $\calL$, let \defcolor{$\pls_\calL(D_B)$} be the closed coamoeba of $\varphi_B(\ini_\calL\calH^c_B)$.
By Proposition~\ref{P:easyTropicalFunctoriality}\eqref{easyTropicalFunctoriality.4}, these are images of the corresponding
components of $\pls(\calH^c_B)$ under the map $\varphi_B$.
More precisely,
\[
   \pls_L(D_B)\ =\ \varphi_B(\pls_L(\calH^c_B)) \qquad\mbox{and}\qquad
   \pls_\calL(D_B)\ =\ \varphi_B(\pls_\calL(\calH^c_B))\,.
\]
We show in Theorem~\ref{Th:main_components} that if $L$ is a non-splitting flat of dimension at least two, then $\pls_L(D_B)$ is a prism
over the closed coamoeba of a reduced discriminant corresponding to the restriction \defcolor{$B|_L$} of $B$ to $L$, and thus has
dimension $d$.
A flacet $L$ of $\calH$ is \demph{essential} if $\dim L>1$ and $0\neq \varphi_B(\omega_L) = \sum\{b\mid b|_L\equiv 0\}$,
as $\omega_L\in \RR^B$ is the indicator function of those $b$ that vanish on $L$.
Theorem~\ref{Thm:PLS_Discriminant} asserts that the $d$-dimensional components of $\pls(D_B)$ consist of the strata
$\pls_L(D_B)$ for essential flacets $L$.

In Section~\ref{Sec:component_identification} we give a coordinate-free description of these objects,  describe the structure of
$\pls_L(D_B)$, and prove Theorem~\ref{Th:main_components}.
We then study the tropical discriminant, its rays, and corresponding strata in Section~\ref{SS:tropDiscr}, where we also
prove Theorems~\ref{Th:conjecture} and~\ref{Thm:PLS_Discriminant}.
The results of this section will be illustrated on our running example in Example~\ref{Ex:facial_discriminants}.

\subsection{Components of $\pls_L(D_B)$}\label{Sec:component_identification}

As in Section~\ref{S:discriminant}, let $\defcolor{A}\subset\ZZ^m$ be $m{+}d$ vectors which span $\ZZ^m$ and  lie on an
affine hyperplane. 
They determine a map $\lambda_A\colon \ZZ^{m+d}\twoheadrightarrow\ZZ^m$ whose kernel \defcolor{$N$} is a free abelian group
of rank $d$. 
The coordinate functions on $\ZZ^{m+d}$ define the embedding $N\hookrightarrow\ZZ^{m+d}$ and are given by $m{+}d$ vectors
\defcolor{$B$} that span the dual $\defcolor{M}\vcentcolon=\Hom(N,\ZZ)$ to $N$.
This set $B$ is a Gale dual to $A$.
Elements of both $A$ and $B$ correspond to the coordinates of $\ZZ^{m+d}$ and are in a canonical one-to-one correspondence.

The torus $\defcolor{\CC^\times_N}\vcentcolon=N\otimes_\ZZ\CC^\times$ has $N$ as its group of cocharacters,
$\Hom(\CC^\times,\CC^\times_N)$, and $M$ as its group of characters, $\Hom(\CC^\times_N,\CC^\times)$.
The groups $M$ and $N$ are naturally dual free abelian groups.
The $d$-dimensional vector spaces $\defcolor{\CC_N}\vcentcolon=N\otimes_\ZZ\CC$ and
$\defcolor{\CC_M}\vcentcolon=M\otimes_\ZZ\CC$ are then dual vector spaces.
For these tensor products of abelian groups,  an integer $k\in\ZZ$ acts on $\CC^\times$ via exponentiation,
$k.t\vcentcolon=t^k$ for
$t\in\CC^\times$ and on $\CC$ via multiplication, $k.\xi\vcentcolon=k\cdot \xi$ for $\xi\in\CC$.

The vectors in $B$ induce a hyperplane arrangement $\calH_B$ in $\CC_N$ (this is $V$ in the notation
of Section~\ref{S:LinearSpace}).
Evaluation at elements of $B$ induces a linear injection
$\lambda_B\colon\CC_N\hookrightarrow\CC^B$ with  $\CC_N{\smallsetminus}\calH_B=\lambda_B^{-1}((\CC^\times)^B)$.
Let $\defcolor{\calH^c_B}\vcentcolon=\lambda_B(\CC_N{\smallsetminus}\calH_B)\subset(\CC^\times)^B$ be the image in
$(\CC^\times)^B$ of the hyperplane complement $\CC_N{\smallsetminus}\calH_B$. 
The multiset $B\subset M$ gives the homomorphism
\[
   \defcolor{\varphi_B}\ \colon\ (\CC^\times)^B\ \ni\ (z_b\mid b\in B)\ \longrightarrow\
   \prod_{b\in B} z_b^b\ \in\ \CC^\times_M\,,
\]
as $M$ is the group of cocharacters of $\CC^\times_M\simeq(\CC^\times)^d$.
Then $\defcolor{D_B}\subset\CC^\times_M$ is the closure of $\varphi_B(\calH^c_B)$.

\begin{remark}\label{R:genDiscr}
  Given any multiset $B$ of nonzero vectors spanning $M$, these definitions of $\lambda_B,\varphi_B,\calH_B$, and $D_B$ make sense.
  Call $D_B$ the \demph{generalized discriminant} associated to $B$.
  If $\sum\{ b\mid b\in B\}=0$, then any Gale dual $A$ to $B$ lies in an affine hyperplane, so $D_B$ is a
  discriminant.
  If $A$ is nondefective configuration then $D_B$ is a hypersurface.
%
%
  \hfill$\diamond$ 
\end{remark}

Let $L\neq\{0\}$ be a flat of the arrangement $\calH_B$.
We study the component of $\pls(D_B)$ corresponding to the image of $\ini_L\calH_B^c$.
By Lemma~\ref{L:calLinitial},  $\ini_L\calH_B^c$ is a product of hyperplane complements in $L$ and in $\CC_N/L$.
We use this to study  its  image under $\varphi_B$ and the coamoeba of that image.
The precise statements are Corollary~\ref{C:factors-fibered} and Theorem~\ref{Th:main_components} below.

Let  $\defcolor{F}\vcentcolon=\{b\in B\mid b|_L\equiv 0\}$ be the flat of the matroid $B$ corresponding to $L$.
Then $L=\calV(F)$.
The flat $F$ spans a rational linear subspace $\defcolor{\QQ_F}\subset \QQ_M$ with corresponding saturated subgroup
$\defcolor{L^\perp}\vcentcolon=\QQ_F\cap M\subset M$.
Note that $L^\perp$ is the saturation of the $\ZZ$-span of $F$ in $M$, and that $F=L^{\perp}\cap B$.
Set $\defcolor{F^c}\vcentcolon=B\smallsetminus F$.

The partition $B=F\coprod F^c$ induces the identification
 \begin{equation} \label{Eq:splitting}
    (\CC^\times)^B\ =\  (\CC^\times)^F\times(\CC^\times)^{F^c}\,.
 \end{equation}
We study how this interacts with the map $\varphi_B\colon (\CC^\times)^B\to \CC^\times_M$.
Note that the inclusion $L^\perp\subset M$ gives the short exact sequence
$L^\perp\hookrightarrow M\twoheadrightarrow M/L^\perp$, which induces a short exact sequence of tori
$\CC^*_{L^\perp}\hookrightarrow \CC^*_M\twoheadrightarrow \CC^*_{M/L^\perp}$.
This is the second row of a commutative diagram with surjective vertical maps and whose top row is split (by~\eqref{Eq:splitting})
\begin{equation}\label{Eq:cDiagram}
   \begin{array}{rcccccl}
     \bOne\ \longrightarrow&(\CC^\times)^F&\longrightarrow&(\CC^\times)^B
     &\xrightarrow{\ \White{\pr}\ }&(\CC^\times)^{F^c}& \longrightarrow\ \bOne
     \\
     & \bigg\downarrow\varphi_F&&\bigg\downarrow\varphi_B&&\bigg\downarrow \varphi_{F^c}&
     \\
     \bOne\ \longrightarrow&\CC^\times_{L^\perp}\hspace{2pt} &\longrightarrow&\CC^\times_M\hspace{4pt}
     &\xrightarrow{\ \defcolor{\pr}\ }&\hspace{4pt}\CC^\times_{M/L^\perp}& \longrightarrow\ \bOne
   \end{array}\ .
 \end{equation}
Here, $\defcolor{\varphi_F}=\varphi_B|_{(\CC^\times)^F}$ and $\defcolor{\varphi_{F^c}}\vcentcolon=\pr\circ\varphi_B|_{(\CC^\times)^{F^c}}$.

As in Section~\ref{S:hyperplane}, $F$ induces a hyperplane arrangement \defcolor{$\calH/L$} in the quotient space $\CC_N/L$.
The linear map $\lambda_F\colon \CC_N\to\CC^F\subset\CC^B$ factors through $\CC_N/L$ and induces an embedding of the complement
$\CC_N/L\smallsetminus \calH/L$ into $(\CC^\times)^F$.
Write \defcolor{$(\calH/L)^c$} for its image in $(\CC^\times)^F$.
Similarly, $F^c$ induces the hyperplane arrangement $\calH|_L$ in $L$.
Restricting the linear map $\lambda_{F^c}\colon \CC_N\to \CC^{F^c}\subset\CC^B$ to $L\subset\CC_N$  embeds the complement
$L\smallsetminus \calH|_L$ into  $(\CC^\times)^{F^c}$.
Write \defcolor{$(\calH|_L)^c$} for its image in $(\CC^\times)^{F^c}$.

By Lemma~\ref{L:calLinitial} applied to the flag $\{0\}\subsetneq L\subsetneq\CC_N$, we have
$\ini_L\calH^c_B=(\calH/L)^c\times (\calH|_L)^c$, which is compatible with the splitting~\eqref{Eq:splitting}.
The image $\varphi_B(\coscrA(\ini_L\calH^c_B))$ of the coamoeba of $\ini_L\calH^c_B$ under the map $\varphi_B$
is a dense subset of $\coscrA(\varphi_B(\ini_L\calH^c_B))$, by
Proposition~\ref{P:easyTropicalFunctoriality}\eqref{easyTropicalFunctoriality.3}.
Let \defcolor{$\pls_L(D_B)$} be the common closure.
Then
\[ 
   \pls_L(D_B)\ =\ 
      \overline{ \varphi_B\bigl(\coscrA(\calH/L)^c \times \coscrA (\calH|_L)^c\bigr)}\,.
\]
We call $\pls_L(D_B)$ the stratum of the phase limit set $\pls(D_B)$ corresponding to $L$.

In the group $(\CC^\times)^B = (\CC^\times)^F\times(\CC^\times)^{F^c}$,
$\ini_L\calH^c_B$ is the product of sets 
\[
   \ini_L\calH^c_B\ =\  (\calH/L)^c\,\times\,(\calH|_L)^c\ =\ 
   \bigl((\calH/L)^c\times\bOne\bigr)\; \cdot\; \bigl(\bOne \times (\calH|_L)^c\bigr)\,.
\]
We use this and the commutative diagram~\eqref{Eq:cDiagram} to study the image $\varphi_B(\ini_L\calH^c_B)$ in $\CC^\times_M$. 
As the span of $F\subset M$ has saturation $L^\perp$ in $M$, the definitions of $\varphi_\bullet$ and $\lambda_\bullet$  imply that
\[
   \varphi_B\bigl((\calH/L)^c \times \bOne\bigr)\ = \
    \varphi_F\bigl((\calH/L)^c \bigr)\ = \ \Cyan{D_F}\ \subset\ \CC^\times_{L^\perp}\ ,
\]
where $D_F$ is the generalized discriminant associated to $F$.
It is generalized as we do not know if $F$ is Gale dual to a nondefective configuration.

Let us write \defcolor{$B|_L$} for the image of $F^c$ in $M/L^\perp$.
This is the multiset of nonzero vectors in the image of $B$ in the dual of $L$ (the restriction of $B$ to $L$), and $B|_L$
induces the hyperplane arrangement $\calH|_L$.
The image of  $\varphi_B( \bOne \times (\calH|_L)^c) \subset \CC^\times_M$
under the map  $\pr\colon \CC^\times_M\to\CC^\times_{M/L^\perp}$ has closure the generalized discriminant \defcolor{$D_{B|_L}$}
associated to $B|_L$.
This is also the closure of $\varphi_{F^c}(\calH|_L)^c$, where $\varphi_{F^c}$ is the rightmost vertical map in~\eqref{Eq:cDiagram}.

\begin{lemma}\label{L:factors}
  The image of $\varphi_B\bigl((\calH/L)^c \times \bOne\bigr)$ in $\CC^\times_{L^\perp}\subset\CC^\times_M$ is the generalized discriminant
  $D_F$ and  the image under $\pr$ of $\varphi_B\bigl( \bOne \times (\calH|_L)^c\bigr)$ in $\CC^\times_{M/L^\perp}$ is
  the discriminant $D_{B|_L}$.
  If $L$ is non-splitting, then $D_{B|_L}$ is a reduced discriminant.
\end{lemma}
\begin{proof}
  The identifications of  $D_F$ and $D_{B|_L}$ were made before the statement of the lemma, except for the observation that  $D_{B|_L}$
  is a discriminant and is reduced when $L$ is non-splitting.
  First, note that $B|_L$ is Gale dual to the subset \defcolor{$A|_L$} consisting of vectors corresponding to $F^c$.
  As $B$ spans $M$, its image $\{\overline{b}\mid \overline{b}\in B|_L\}$ spans $M/L^\perp$.
  Also, as $\sum\{b\mid b\in B\}=0$ its image also has sum zero, and thus $D_{B|_L}$ is a discriminant.

  Suppose that $L$ is a non-splitting flat.
  Then $L$ lies in a non-splitting flag $\calL$ of flats satisfying~\eqref{Eq:CC_char}
  that is a witness for $A$ to be nondefective.
  The image of the flag $\calL$ in $M/L^\perp$ (its restriction to $L$) gives a non-splitting flag as it satisfies
  the condition~\eqref{Eq:CC_char}.
  Thus  $A|_L$ is nondefective and $D_{B|_L}$ is a reduced discriminant.
\end{proof}

We study the map $\pr\colon\varphi_B(\ini_L\calH^c_B) \to D_{B|_L}$ that is the restriction of
$\pr\colon\CC^\times_M\to\CC^\times_{M/L^\perp}$.

\begin{lemma}\label{L:fibered}
  Under the map $\pr\colon\CC^\times_M\to \CC^\times_{M/L^\perp}$, the fiber of  $\varphi_B(\ini_L\calH^c_B)$ over a point
  $z\in D_{B|_L}$ is the union of translates of $D_F$ by $\varphi_B(\bOne\times \{y\})$ for those $y\in(\calH|_L)^c$ with image $z$.
\end{lemma}
\begin{proof}
  Observe that
  \[
  \varphi_B(\ini_L\calH^c_B)\ =\ 
  \varphi_B\bigl((\calH/L)^c \times (\calH|_L)^c\bigr)  \ =\ D_F\; \cdot\;
  \varphi_B\bigl(\bOne\times (\calH|_L)^c\bigr)\,.
  \]
  As $D_F$  lies in the kernel $\CC^\times_{L^\perp}$ of $\pr$~\eqref{Eq:cDiagram}, we see that
  $\pr^{-1}(z)\cap \varphi_B(\ini_L\calH^c_B)$ is the translation of $D_F$ by points
  $y\in \pr^{-1}(z)\cap \varphi_B\bigl(\bOne\times (\calH|_L)^c\bigr)$.
  This proves the lemma.
\end{proof}

The set $\varphi_B(\coscrA(\ini_L\calH^c_B))\subset \TT_M$ is a dense
subset of the stratum $\pls_L(D_B)$ of the phase limit set of $D_B$ corresponding to the flat $L$ of $\calH_B$.
The commutative diagram~\eqref{Eq:cDiagram} of exact sequences of tori restricts to a commutative diagram of exact
sequence of compact subtori. 
Lemmas~\ref{L:factors} and~\ref{L:fibered} have the following consequence for coamoebas.

\begin{corollary}\label{C:factors-fibered}
  Let $L$ be a flat of $\calH_B$.
  Then $\overline{\varphi_B(\coscrA((\calH/L)^c))}=\overline{\coscrA(D_F)}$ and 
  \[
     \overline{\pr(\varphi_B(\coscrA(\ini_L\calH^c_B)))}\ =\ 
     \overline{\coscrA(D_{B|_L})}\ \subset\ \TT_{M/L^\perp}\,.
  \]
  The fiber of $\pr$ over $\theta\in \coscrA(D_{B|_L})$ is the union of the translates
  $\coscrA(D_F) \cdot \vartheta$ for all $\vartheta\in \varphi_B(\bOne\times \coscrA(\calH|_L)^c)$ with
  $\pr(\vartheta)=\theta$.
\end{corollary}

We show that $\pls_L(D_B)$  is a prism when $L$ is non-splitting and $\dim L > 1$.

\begin{theorem}\label{Th:main_components}
 Suppose that $A$ is nondefective with reduced discriminant $D_B$.
 Let $L$ be a non-splitting flat of $\calH_B$.
 Then the stratum $\pls_L(D_B)$ is the prism over the closed coamoeba of $D_{B|_L}$.
 When $\dim L>1$, this has dimension $d$.
\end{theorem}  

As we show in Lemma~\ref{L:type_1} below, if $\dim L=1$, then $\dim\pls_L(D_B)< d$.

\begin{proof}
  First suppose that $L$ is a hyperplane of $\calH_B$.
  Then the corresponding flat $F$ of $B$ is the parallel-equivalence
  class $[b]$ of any $b\in B$ that defines $L$, and $L^\perp\simeq\ZZ$.
  Since $L$ is a hyperplane, $(\calH/L)^c=(\CC_N/L)^\times\simeq\CC^\times$.
  As $L$ is non-splitting, $0\neq \sum_{c\in[b]}c$, and  by  Remark~\ref{R:parallel} $\varphi_F$ is not a
  constant map, which shows that $D_F=\CC^\times_{L^\perp}\simeq\CC^\times$.
  Thus $\coscrA(D_F)=\TT_{L^\perp}$ and by Corollary~\ref{C:factors-fibered},
  $\pls_L(D_B) = \pr^{-1}(\overline{\coscrA(D_{B|_L})})$ is a prism.
  Consequently, when $\dim L>1$, so that $d>2$, this has dimension $1+\dim\coscrA(D_{B|_L})=1+(d{-}1)=d$.

  Suppose now that $L$ is not a hyperplane.
  Let $H$ be the hyperplane of $\calH_B$ in some non-splitting flag containing $L$.
  Then $H$ is a non-splitting hyperplane.
  Consider the flag $\defcolor{\calL}\colon \{0\}\subset L \subset H\subset \CC_N$.
  Since $\ini_{\calL}\calH_B^c$ is an initial scheme of both $\ini_L\calH_B$ and $\ini_H\calH^c_B$, we have the containment
  of phase limit sets, 
  $\pls_\calL(D_B)\subset\pls_L(D_B)$ and  $\pls_\calL(D_B)\subset\pls_H(D_B)$.
  We will use induction on the dimension $d$ and the case of the hyperplane $H$ to show that $\pls_\calL(D_B)$ is the prism
  over the closure of $\coscrA(D_{B|_L})$, which implies that $\pls_L(D_B)=\pls_\calL(D_B)$ and is thus also the same prism.
  The assertion on dimension follows from Corollary~\ref{Co:dim_Discr_Coamoeba}.

  Let $G\subset F\subset B$ be the flag of flats in $B$ corresponding to $L\subset H$.
  We have that $G=[b]$, the parallel-equivalence class of any element $b\in B$ that defines the hyperplane $H$.
  Then $B|_H$ is the image of $B{\smallsetminus}G$ in $M/H^{\perp}$.
  Restricting a non-splitting flag containing $L$ and $H$ to $H$ shows that $L$ is a non-splitting flat of the hyperplane arrangement
  $\calH_{B|_H}$.
  Note that $B|_L=(B|_H)|_L$.
  Applying the induction hypothesis, $\pls_L(D_{B|_H})$ is the prism in $\TT_{M/H^\perp}$ over closed coamoeba of
  $D_{B|_L}$ (which is a subset 0f $\TT_{M/L^\perp}$), and if $\dim L>1$, then it has dimension $d{-}1$.

  By Lemma~\ref{L:calLinitial} and~\eqref{Eq:calH(calL)}, and again using that $B|_L=(B|_H)|_L$, 
  \[
  \ini_\calL\calH_B^c\ =\ 
  \left(\calH_B/H\right)^c  \times\left( (\calH_{B|_H})/L\right)^c\times \left(\calH_{B|_L}\right)^c\,,
  \]
  which  is $ (\calH_B/H)^c \times\ini_L(\calH_{B|_H})^c$.
  Since $\pls_\calL(D_B)$ is the closure of $\coscrA(\varphi_B(\ini_\calL\calH_B^c))$, we see that the image of
  $\pls_\calL(D_B)\subset\TT_M$ in $\TT_{M/H^\perp}$ is $\pls_L(D_{B|_H})$.
  By Corollary~\ref{C:factors-fibered}, the fiber contains a translate of $\coscrA(D_G)=\TT_{H^\perp}\simeq\TT$, as
  $D_G=\CC^\times_{H^\perp}\simeq\CC^\times$.
  Thus
  \[
  \pls_\calL(D_B)\ =\ \pr^{-1}\left(\pls_L(D_{B|_H})\right)\,,
  \]
  so that it is the prism over $\pls_L(D_{B|_H})$, which is itself the prism over the closed coamoeba of
  $D_{(B|_H)|_L}=D_{B|_L}$. 
  Since a prism over a prism is a prism, we conclude that $\pls_\calL(D_B)$ is the prism over the closed coamoeba of $D_{B|_L}$,
  which completes the proof of the theorem.
\end{proof}

\subsection{The phase limit set of a reduced discriminant}\label{SS:tropDiscr}

The tropical variety of $D_B$ is the \demph{tropical discriminant} \defcolor{$\trop(D_B)$}~\cite{DFS}.
This is the image of the Bergman fan $\Berg(B)$ under the linear map
$\varphi_B\colon \RR^B\to\RR_M$ given by 
 \[
   \varphi_B\ \colon\ (x_b\mid b\in B)\ \longmapsto\ \sum x_b\, b\,.
 \]
Since $\sum \{b\mid b\in B\}=0$, the subspace $\RR\bOne_B$ of the lineality space of $\Berg(B)$ lies in the kernel of
$\varphi_B$ so that the map to the tropical discriminant factors through $\Berg(B)/\RR\bOne_B$.

By Lemma~\ref{L:rays}, $\pls(D_B)$ is the union of strata arising from initial
schemes of $D_B$ given by the rays in any fan structure on the tropical discriminant.
Before studying the tropical discriminant in more detail, we consider our running example.

\begin{example} \label{Ex:Tdiscr}
 Figure~\ref{F:tdiscr} shows two views of the tropical discriminant $\trop(D_B)\subset\RR^3$ for the vectors of
 Example~\ref{Ex:RunningI}.
 \begin{figure}[htb]
 \centering
  \begin{picture}(195,158)(-11,0)
    \put(0,0){\includegraphics[height=148pt]{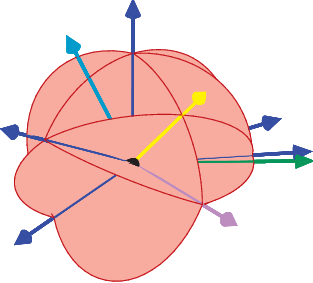}}
    \put(70,149){$b_3$}
    \put(23,134){$b_{236}$}  
       \thicklines \put(103,125){\White{\line(0,-1){20}}} \thinlines
     \put(100,130){$\rho$}\put(103,125){\vector(0,-1){23}}
    \put(-11,79){$b_6$}           \put(138,90){$b_2$}  \put(157,75){$b_4$} 
             \put(110,20){$b_1$}  
      \put(149,48.5){$b_{124}$} 
    \put(0,10){$b_5$}
    \end{picture} \qquad
    \begin{picture}(186,158)(-20,0)
    \put(0,0){\includegraphics{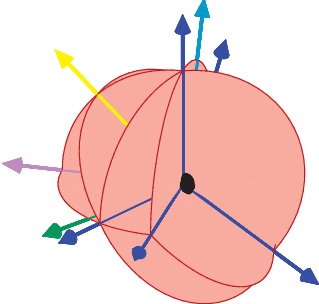}}
    \put(73,138){$b_3$}      \put(93,151){$b_{236}$} 
    \put(20,125){$\rho$}    \put(113,120){$b_6$} 

    \put(-2,74){$b_1$} 
     \put(12,42){$b_{124}$}  
     \put(29,18.5){$b_4$} 
     \put(156,10){$b_5$}
     \thicklines
        \put(48.7,5.3){\White{\vector(1,1){13}}}
        \put(48.3,5.7){\White{\vector(1,1){15.1}}}
        \put(48,6){\White{\vector(1,1){15.5}}}
        \put(47.7,6.3){\White{\vector(1,1){15.1}}}
        \put(47.3,6.7){\White{\vector(1,1){13}}}
     \thinlines
    \put(40,0){$b_2$} \put(48,6){\vector(1,1){15}}
 \end{picture}
%
  \caption{Two views of the tropical discriminant $\trop(D_B)$.}
  \label{F:tdiscr}
  \end{figure}
  Rays are labeled by spanning vectors.
  Among them are the six rays $\RR_{\geq 0} b_i$ for $i=1,\dotsc,6$ corresponding to the vectors $b_i$, which
  are the images of the rays $\RR_{\geq 0}e_i$ of $\Berg(B)$.
  Recall from Figure~\ref{F:edge_graph} that the Bergman fan $\Berg(B)$ has two more rays spanned by the vectors
  $e_{124}$ and $e_{236}$. 
  Write \defcolor{$b_{124}$} and \defcolor{$b_{236}$} for their images under $\varphi_B$.
  Notice that
  \[
    b_{124}\ \vcentcolon=\ \varphi_B(e_{124})\ =\ b_1+b_2+b_4\ =\ (2,3,0)\,,
  \]
  and the cones from $\Berg(B)$ meeting $e_{124}$ (spanned by $e_{124}$ together with one of $e_1$, $e_2$, and $e_4$) map onto the two
  coplanar cones spanned by $b_1,b_4$ and by $b_4,b_2$, with $b_{124}$ lying in the first.
  Similarly, the three cones of $\Berg(B)$ meeting $e_{236}$ have coplanar image with $b_{236}$ lying in the cone spanned by
  $b_3$ and $b_6$.

  Any fan structure on $\trop(D_B)$ contains the ray $\defcolor{\rho}$ through the point $(1,0,1)$, which is the
  intersection of the images of the cones $\sigma_{13}$ and $\sigma_{46}$. 
  We will see that coamoebas of initial schemes for $\rho$, $b_{124}$,  and $b_{236}$ all lie in the
  union of other components of $\pls(D_B)$. \hfill$\diamond$
\end{example} 

Let us consider fan structures on $\trop(D_B)$ such that for every cone $\sigma$ of $\Berg(B)$, $\varphi_B(\sigma)$ is a
union of cones of $\trop(D_B)$.
Fix one such fan structure on $\trop(D_B)$, and let $\rho$ be one of its rays.
Either $\rho$ is the image of a ray of $\Berg(B)$ or it is not.
Recall that the rays of $\Berg(B)$ correspond to flacets of $\calH_B$.
The ray corresponding to a flacet $L$ is generated by the indicator function $\omega_L\in \RR^B$ of $L$.
Let $\defcolor{b_L}\vcentcolon=\varphi_B(\omega_L)$, which is the sum of the elements $b$ of $B$ that vanish along $L$.
When $b_L\neq 0$, the image of this ray of $\Berg(B)$ is a ray of $\trop(D_B)$.
A consequence of this discussion is the following lemma.

\begin{lemma}\label{L:Ray_classification}
  Rays $\rho$ of $\trop(D_B)$ are one of the following two types.
\begin{enumerate}
  \item  $\rho=\RR_{\geq 0}b_L$, where $L$ is a flacet of $\calH_B$ with $b_L\neq 0$.
  \item  $\rho$ is not the image of a ray of $\Berg(B)$.
\end{enumerate}
   
\end{lemma}

Rays $\rho$ of $\trop(D_B)$ have type  (1) or (2), according to this classification.

\begin{lemma}\label{L:type_2}
  If a ray $\rho$ of  $\trop(D_B)$ has type $(2)$, then the component $\pls_\rho(D_B)$ lies in the union of
  components $\pls_\tau(D_B)$ for $\tau$ a ray of type $(1)$. 
\end{lemma}
\begin{proof}
 Let $\sigma$ be a cone of $\Berg(B)$ such that $\rho\subset\varphi_B(\sigma)$.
 Then $\sigma$ is not a ray, as $\rho$ has type (2).
 Let $\tau\subset\sigma$ be a ray of $\sigma$.
 Then $\tau$ corresponds to a flacet $L$ of $\calH_B$.
 As  $\ini_\sigma \calH^c_B$ is an initial scheme of $\ini_\tau \calH^c_B=\ini_L\calH^c_B$,
 we have $\coscrA(\ini_\sigma\calH^c_B) \subset \overline{\coscrA(\ini_L\calH^c_B)}$, and thus
 $\pls_\sigma(\calH^c_B)\subset\pls_L(\calH^c_B)$.
 Applying $\varphi_B$, shows that
  \begin{equation}\label{Eq:pls_containment}
    \varphi_B(\pls_\sigma(\calH^c_B))\ \subset\ \varphi_B(\pls_L(\calH^c_B))\,.
  \end{equation}
 Suppose that the flacet $L$ corresponding to $\tau$ has $\varphi_B(\omega_L)=0$.
 Since $\omega_L$ spans $\tau$, $\varphi_B(\tau)=0$.
 As $\rho\subset \varphi_B(\sigma)$, and the cone $\sigma$ is generated by its rays, it has a ray $\tau$ with  $\varphi_B(\tau)\neq 0$.
 If $L$ is the flacet corresponding to the ray $\tau$, then we have~\eqref{Eq:pls_containment}.

 Lemma~\ref{L:RayLemma} applied to the homomorphism $\varphi_B$
 implies that $\pls_\rho(D_B)$ is the union of images $\varphi_B(\pls_\sigma(\calH^c_B))$ for cones $\sigma$ of $\Berg(B)$
 such that $\rho\subset\varphi_B(\sigma)$.
 Each such cone $\sigma$ has a ray corresponding to a flacet $L$ with  $\varphi_B(\omega_L)\neq 0$,
 and for this flacet we have~\eqref{Eq:pls_containment}.
 This completes the proof.
\end{proof}

\begin{lemma}\label{L:type_1}
  Let $L$ be a one-dimensional flat of $\calH_B$.
  Then the stratum $\pls_{L}(D_B)$ corresponding to $L$ has dimension at most $d{-}1$.  
\end{lemma}
\begin{proof}
  By Corollary~\ref{C:factors-fibered}, if $\pr\colon\CC^\times_M\to \CC^\times_{M/L^\perp}\simeq\CC^\times$, then
  $\pr(\pls_L(D_B))\subset\coscrA(D_{B|_L})$.
  As $L$ is 1-dimensional, $D_{B|_L}$ and its coamoeba are points.
  Thus $\pls_L(D_B)$ is a subset of a fiber of $\pr$ and thus has  dimension at most $d{-}1$.  
\end{proof}

A flacet $L$ is \demph{essential} if $\dim L > 1$ and $0\neq b_L = \varphi_B(\omega_L)$.
Let \defcolor{$\TOP(D_B)$} be the union of $d$-dimensional components of $\pls(D_B)$.
We have proven the following result.

\begin{theorem}\label{Thm:PLS_Discriminant}
  When $D_B$ is a reduced discriminant, we have $\TOP(D_B)=\bigcup\pls_L(D_B)$, the union over essential flacets $L$ of $\calH_B$.
\end{theorem}
\begin{proof}[Proof of Theorem~\ref{Th:conjecture}]
  When $d=3$, the essential flacets of $\calH_B$ are hyperplanes $H$ with $b_H\neq 0$.
  All are non-splitting.
  Theorem~\ref{Th:conjecture} follows from Theorems~\ref{Thm:PLS_Discriminant} and~\ref{Th:main_components}.
\end{proof}

\begin{example}\label{Ex:facial_discriminants}
  We end with our running example, illustrating some results of this section.
  Figure~\ref{F:polytopes} shows three views of the Newton polytope of the discriminant~\eqref{Eq:bigDiscrim}.
  \begin{figure}[htb]
    \centering
    \includegraphics[height=90pt]{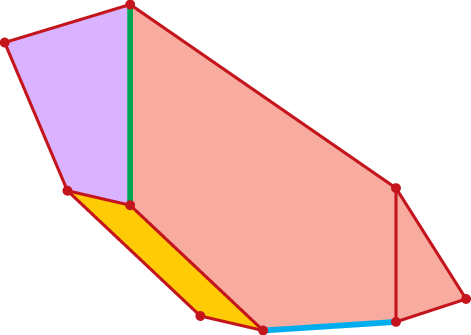}\qquad
    \includegraphics[height=100pt]{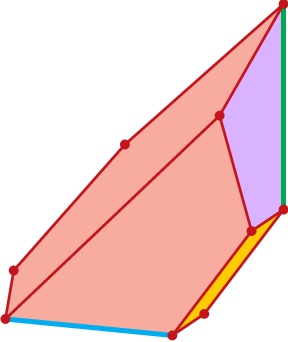}\qquad
    \includegraphics[height=100pt]{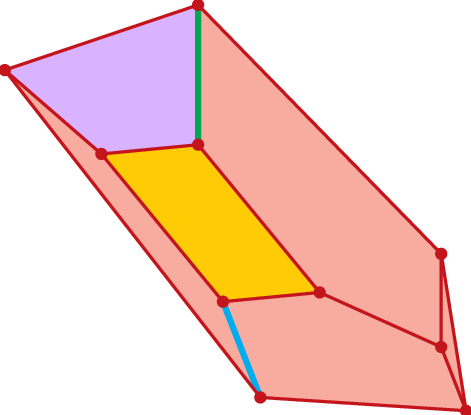}

    \caption{Three views of the Newton polytope of the discriminant.}
    \label{F:polytopes}
  \end{figure}
  Each cone of the tropical discriminant of Figure~\ref{F:tdiscr} consists of inner normals to a face of the Newton polytope, and 
  faces corresponding to the rays $b_1$, $b_{124}$, $b_{236}$, and $\rho$ are color-coded to correspond to Figure~\ref{F:tdiscr}.
  An initial scheme of the discriminant is defined by the restriction of the discriminant to the corresponding face, which
  we call the \demph{initial form}.
  We give these for these four rays.
  \begin{eqnarray*}
    \ini_{b_1}D_B &=& q^2r^2(3125q^2r^2-1024r^3+4000qr^2+768r^2-200qr-192r+16q+16)\\
   \ini_{b_{124}}D_B &=& 16q^2r^2 - 1024q^2r^5\ =\ 16q^2r^2(1-64r^3)\\
   \ini_{b_{236}}D_B &=& 16p^2q^2 - 1024p^5q^2\ =\ 16p^2q^2(1-64p^3)\\
    \ini_\rho D_B &=& 16q^3r^2 + 16q^2r^2 + 16p^2q^2 + 16p^2q\ =\ 16q(q+1)(qr^2+p^2)
 \end{eqnarray*}
  %
%

%
  Rays $b_{124}$ and $b_{236}$ are images of rays of $\Berg(B)$ corresponding to one-dimensional flacets
  (points $p_{124}$ and $p_{236}$ in Figure~\ref{F:RunningExampleI}), which are not essential.
  Their initial forms are binomials, which implies that their coamoebas are 2-dimensional.
  This illustrates Lemma~\ref{L:type_1}.

  The ray $\rho$ occurs as the intersection of the images $\varphi_B(\sigma_{13})$ and $\varphi_B(\sigma_{46})$ of two
  2-dimensional cones of $\Berg(B)$, subdividing both.
  The ray $\rho$ exposes the parallelogram facet of the Newton polytope of $D_B$.
  Its  edges are exposed by the subdivided cones, with the two parts of each exposing parallel edges,
  and the corresponding initial forms for the parallel edges differ only by a monomial factor.
  The initial form $\ini_\rho D_B$ factors (up to monomials) as the product of forms for each pair of parallel edges.
  As $\ini_\rho D_B$ is the product of binomials, its coamoeba is two-dimensional.
  This illustrates Lemma~\ref{L:type_2}, particularly its proof.

  Consider $\ini_{b_1}D_B$.
  Let $\defcolor{H}\in\calH_B$ be the hyperplane defined by $b_1$.
  Then $B|_H$ consists of the rows of the matrix below.
  We also give the parametrization $\psi\colon \CC^2\smallsetminus\calH_{B|_H}\to(\CC^\times)^2$ of $D_{B|_F}$.
  \[
  B|_H\ =\ B/b_1\ =\
  \left( \begin{array}{cc}1&0\\0&1\\2&0\\-1&-2\\-2&1\end{array}\right)
    \qquad
    (y,z)\ \stackrel{\psi}{\longmapsto}\
    \left( \frac{y(2y)^2}{(-y-2z)(-2y+z)^2}\,,\, \frac{ z(-2y+z)}{(-y-2z)^2} \right)\ .
  \]
  In coordinates $q,r$ for $\CC^2$, the reduced discriminant $D_{B|_H}$ has implicit equation
  \[
     D_{B|_H}\ =\ 3125q^2r^2-1024r^3+4000qr^2+768r^2-200qr-192r+16q+16\,,
  \]
  which is the non-monomial factor of $\ini_{b_1} D_B$.

  Consider the Newton polytope $P_H$ of  $D_{B|_H}$ and the configuration $F_H$, where we replace parallel vectors of $B|_H$
  by their sum.
  By Remark~\ref{R:parallel}, this does not change the coamoeba.
  \begin{equation}\label{Eq:F_H}
   \raisebox{5pt}{\includegraphics{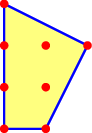}}
  \qquad
  \begin{picture}(128,74)(-14,-5)
    \put(0,0){\includegraphics{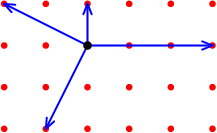}}
    \put( 46,60){$f_1$}
    \put(105,39){$f_3$}
    \put(  9,-2){$f_2$}
    \put(-11,60){$f_4$}
  \end{picture}
  \end{equation}
 While $0=f_1+\dotsb+f_4$, note that there are other integer linear relations they satisfy,
 \[
 6 f_1 + 3 f_2 +  f_3 + 0\cdot f_4\ =\ 0\cdot f_1 + 3 f_2 + 5 f_3 + 6 f_4\ =\ 0 \,.
 \]
 Let $\defcolor{d_H}=6$, which is the length (area) of the convex hull of any Gale dual of $F_H$; this is the difference of
 the minimum and maximum  coefficient in the these two relations.
  
 The coamoeba of a discriminant when $d=3$ is described in~\cite{DFS,PS}.
 We explain this for the coamoeba of $D_{F_H}$.
 It is most propitious to view this in the universal cover of $\TT^2=\RR^2/(2\pi\ZZ)^2$, where the coamoeba is the complement of
 a zonotope in a covering of $\TT^2$.

 We let \defcolor{$Z_H$} be the zonotope of $F_H$---this is the sum of intervals $[{\bf 0},\pi b]$ for $b\in F_H$.
 Its normal fan is the hyperplane arrangement $\calH_{F_H}$.
 We display it as an oriented cycle.
  \begin{equation}\label{Eq:OrientedZonotope}
    \raisebox{-28pt}{\includegraphics{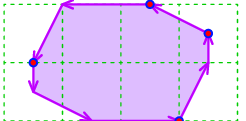}}
  \end{equation}
 The closed coamoeba $\coscrA(D_{F_H})$ is the closure of the image of $(\CC\smallsetminus\RR)\subset\PP^1$ under the map $\Arg\circ\psi$,
 which has no critical points in $\CC\smallsetminus\RR$.
 As with the Zonotope $Z_H$, it is best to lift this to the universal cover $\RR^2$ of $\TT^2$.
 Let \defcolor{$\coscrA^+_H$} be the closure of the image of the upper half plane under the lift of $\Arg\circ\psi$, and
 \defcolor{$\coscrA^-_H$} the same for the lower half plane.
 We consider this as an oriented cycle.

 We describe $\coscrA^+_H$ by describing its oriented boundary.
 Let $v$ be a vertex of $Z_H$ whose incoming oriented edge is $\pi f$, for some $f\in F_H$.
 In~\eqref{Eq:OrientedZonotope}, the four choices for $v$ are the marked vertices.
 Each vector of $F_H$ spans a line, and let $f_1=f, f_2,\dotsc,f_r$ be the ordering of the vectors of $F_H$ corresponding
 to the order in which the lines they support are encountered moving clockwise from $f$.
 In our example, we choose $v$ to the northeasternmost vertex of $Z_H$---for this choice we get the ordering of $F_H$ indicated
 in~\eqref{Eq:F_H}.
 Then the oriented boundary of $\coscrA^+_H$ is, in order 
 \[
    v\,,\
    v-\pi f_1\,,\ 
    v-\pi f_1 -\pi f_2\,,\ \dotsc\,,\ 
    v-\pi f_1- \dotsb - f_{r-1}\,,\  v\,.
 \]
 Let $\coscrA^-_H$ be the reflection of $\coscrA^+_H$ in the origin.
 We show $\coscrA^+_H$ and the union of the cycles $Z_H$,   $\coscrA^+_H$, and  $\coscrA^i_H$.
 \[
   \raisebox{40pt}{\includegraphics{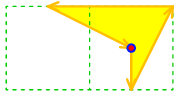}}\qquad
   \includegraphics{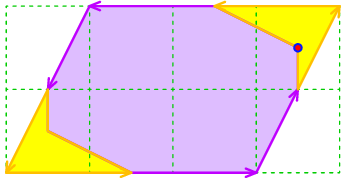}
 \]
 Observe that   $Z_H \cup \coscrA^+\cup \coscrA^-_H$ covers $d_H=6$ fundamental domains of $\TT^2$ in $\RR^2$.
 We summarize the major results of~\cite{DFS,PS} about the coamoeba discriminant.
 \begin{enumerate}
 \item
   The closed coamoeba of $D_{B|_H}$ is the image of  $\coscrA^+_H\cup \coscrA^-_H$ in $\TT^2$.
 \item
   The oriented cycle $Z_H \cup \coscrA^+\cup \coscrA^-_H$ pushes forward to $\TT^2$ as the cycle $d_H[\TT^2]$.   \hfill$\diamond$ 
 \end{enumerate}
\end{example}  

\providecommand{\bysame}{\leavevmode\hbox to3em{\hrulefill}\thinspace}
\providecommand{\MR}{\relax\ifhmode\unskip\space\fi MR }
\providecommand{\MRhref}[2]{%
  \href{http://www.ams.org/mathscinet-getitem?mr=#1}{#2}
}
\providecommand{\href}[2]{#2}

\end{document}